\documentclass[a4paper,reqno,11pt]{amsart} 
\usepackage[T1]{fontenc}
\usepackage[utf8]{inputenc}

\usepackage{amsmath, amssymb, graphicx, enumerate, enumitem, color}
\usepackage[dvipsnames,svgnames,table]{xcolor}
\usepackage{svg}
\usepackage[foot]{amsaddr}
\usepackage[textsize=tiny]{todonotes}

\usepackage{multirow}
\usepackage{longtable}
\usepackage{booktabs}
\usepackage{comment}
\usepackage{attachfile2}

\usepackage[normalem]{ulem}

\usepackage{listings}
\usepackage[T1]{fontenc}
\usepackage{xcolor}
\usepackage{textcomp}

\definecolor{commentscolor}{rgb}{0,0.6,0}
\definecolor{keywordscolor}{HTML}{0512FF}
\definecolor{stringscolor}{HTML}{AB354D}
\definecolor{backcolor}{HTML}{FFFFFF}
\definecolor{framecolor}{HTML}{CCCCCC}

\lstdefinestyle{mystyle}{
    frame=single,
    framesep=5pt,
    rulecolor=\color{backcolor},
    upquote=true,
    backgroundcolor=\color{backcolor},   
    commentstyle=\color{commentscolor},
    keywordstyle=\color{keywordscolor},
    stringstyle=\color{stringscolor},
    basicstyle=\ttfamily,
    breakatwhitespace=false,         
    breaklines=true,                 
    captionpos=b,                    
    keepspaces=true,                                    
    showspaces=false,                
    showstringspaces=false,
    showtabs=false,                  
    tabsize=4
}
 
\makeatletter
\def\lst@outputspace{{\ifx\lst@bkgcolor\empty\color{white}\else\lst@bkgcolor\fi\lst@visiblespace}}
\makeatother
\definecolor{1}{HTML}{FF0303}
\definecolor{2}{HTML}{2323FF}
\definecolor{3}{HTML}{CB9154}
\definecolor{4}{HTML}{6AA84F}
\definecolor{5}{HTML}{BA54FF}
\definecolor{6}{HTML}{A9537E}
\definecolor{7}{HTML}{ED7E0F}

\newcommand{\1}{\color{1}1\color{black}}
\newcommand{\2}{\color{2}2\color{black}}
\newcommand{\3}{\color{3}3\color{black}}
\newcommand{\4}{\color{4}4\color{black}}
\newcommand{\5}{\color{5}5\color{black}}
\newcommand{\6}{\color{6}6\color{black}}
\newcommand{\7}{\color{7}7\color{black}}
\newcommand{\s}{\hphantom{1}}

\newcommand{\defin}[1]{\emph{\textcolor{ForestGreen}{#1}}}

\makeatletter
\def\thm@space@setup{
  \thm@preskip=4mm
  \thm@postskip=0mm
}
\makeatother

\usepackage{tcolorbox}
\usepackage{hyperref}
\usepackage{xcolor}
\hypersetup{
  pdfauthor={Jędrzej Hodor, Jakub Sordyl},
  pdftitle={Local dimension of a Boolean lattice},
    colorlinks,
    linkcolor={RoyalBlue},
    citecolor={RubineRed},
    urlcolor={blue!80!black}
}
\usepackage[capitalise, noabbrev]{cleveref}
\usepackage{mathtools}

\theoremstyle{plain} 
\newtheorem{theorem}{Theorem}

\newtheorem{conjecture}[theorem]{Conjecture}
\newtheorem{lemma}[theorem]{Lemma}

\theoremstyle{remark} 
 
\bibstyle{plain}

\newcommand{\lat}[1]{\mathcal{B}_{#1}}
\newcommand{\mlat}[1]{\mathcal{M}_{#1}}
\newcommand{\mn}[2]{\mathcal{M}_{#1, #2}}

\newcommand{\cgA}{\mathcal{A}} 
\newcommand{\calA}{\mathcal{A}}

\newcommand{\calL}{\mathcal{L}}

\newcommand{\calM}{\mathcal{M}}

\newcommand{\calS}{\mathcal{S}} 
\newcommand{\cgS}{\mathcal{S}}

\newcommand{\ldim}{\operatorname{ldim}}

\newcommand{\bdim}{\operatorname{bdim}}

\graphicspath{{Figures/}}

\let\leq\leqslant
\let\geq\geqslant

\let\subset\subseteq

\let\epsilon\varepsilon

\let\setminus\backslash

\renewenvironment{enumerate}{\begin{enumorig}[label=\textup{(\roman*)}, noitemsep, 
topsep=2pt plus 2pt, labelindent=.2em, leftmargin=*, widest=iii]}{\end{enumorig}}

\makeatletter 
\let\old@setaddresses\@setaddresses 
\def\@setaddresses{\bigskip\bgroup\parindent 0pt\let\scshape\relax\old@setaddresses\egroup}
\makeatother

\begin{document} 
\title[Local dimension of a Boolean lattice] 
{Local dimension of a Boolean lattice}

\email{jedrzej.hodor@gmail.com}
\email{jakubsordyl1@gmail.com}

\author[J.~Hodor]{Jędrzej Hodor}
\address[J.~Hodor]{Theoretical Computer Science Department, 
Faculty of Mathematics and Computer Science and  Doctoral School of Exact and Natural Sciences, Jagiellonian University, Krak\'ow, Poland}

\author[J.~Sordyl]{Jakub Sordyl}
\address[J.~Sordyl]{Theoretical Computer Science Department\\ 
  Jagiellonian University\\
  Kraków, Poland}

\thanks{J.~Hodor is supported by the National Science Center of Poland under grant UMO-2022/47/B/ST6/02837 within the OPUS 24 program.}


\begin{abstract}
    For every integer $n$ with $n \geq 4$, we prove that the local dimension of a poset consisting of all the subsets of $\{1,\dots,n\}$ equipped with the inclusion relation is strictly less than $n$, answering a question of Kim, Martin, Masařík, Shull, Smith, Uzzell, and Wang (Eur.\ J.\ Comb.\ 2020).
    We also study several related problems.
\end{abstract}

\maketitle

\section{Introduction}\label{sec:introduction}
A key measure of complexity of partially ordered sets is dimension introduced in 1941 by Dushnik and Miller~\cite{DM41}.
Throughout the years, many variants of dimension emerged.
In this paper, we study local dimension, which was proposed in 2016 by Ueckerdt~\cite{Ueckerdt16}.
Let us start by defining the above concepts.

\subsection{Definitions}
A \defin{partially ordered set}, or \defin{poset} for short, is an ordered pair $P = (X, \leq)$, where $X$ is a non-empty set of elements called the \defin{ground set} of $P$, and $\leq$ is a binary relation on $X$ (called the \defin{order relation} in $P$), which is reflexive, antisymmetric and transitive. 
We do not require ground sets to be finite.
For two posets $P$ and $Q$, we say that $Q$ is a \defin{subposet} of $P$ (denoted by $Q \subseteq P$) if the ground set of $Q$ is a subset of the ground set of $P$ and the order relation of $Q$ is the restriction of the order relation in $P$ to the ground set of $Q$.
Let $P$ be a poset.
Sometimes, we replace the phrase $x \leq y$ in $P$ with \defin{$x \leq_P y$}. 
We say that two elements $x,y \in P$ are \defin{comparable} if $x \leq_P y$ or $y \leq_P x$. 
A poset, where all pairs of elements are comparable is called a \defin{linear order}. 
A poset $L$ is a \defin{linear extension} of $P$ if $L$ is a linear order on the ground set of $P$ such that $x \leq_L y$ whenever $x \leq_P y$ for every two elements $x$ and $y$ in $P$. 
A \defin{partial linear extension} of $P$ is a linear extension of a subposet of $P$. 
A \defin{local realizer} of $P$ is a family $\mathcal{L}$ of partial linear extensions of $P$ such that for all $x,y \in P$, there exists $L \in \mathcal{L}$ with $x,y \in L$ and
 \[ x  \leq_P y  \text{ if and only if } \text{there is no } L \in \mathcal{L} \text{ with }y <_L x.\]
 The \defin{size} of a local realizer $\mathcal{L}$ of $P$ is the number of elements of $\mathcal{L}$.
A \defin{realizer} of $P$ is a local realizer of $P$ where every element is a linear extension of $P$.
The \defin{dimension} of $P$ is the minimum size of a realizer of $P$, and is denoted by \defin{$\dim(P)$}.
 The \defin{frequency} of an element $x \in P$ in a local realizer $\mathcal{L}$ of $P$ is the number of $L \in \mathcal{L}$ with $x \in L$ and the \defin{frequency} of $\mathcal{L}$ is the maximum frequency of an element of $P$ in $\mathcal{L}$.
The \defin{local dimension} of $P$ is the minimum frequency of a local realizer of $P$, and is denoted by \defin{$\ldim(P)$}. 
By definition, $\ldim(P) \leq \dim(P)$.

\subsection{Local dimension of a Boolean lattice}
For a positive integer $n$, the \defin{Boolean lattice} of \defin{order} $n$, denoted by \defin{$\lat{n}$}, is the poset on all the subsets of $[n] = \{1,\dots,n\}$ ordered by the inclusion relation.
Basics of the posets' dimension theory yield that $\ldim(\lat{n}) \leq \dim(\lat{n}) = n$.
Note however, that in some sense the full complexity of a Boolean lattice is not necessary for $\dim(\lat{n}) = n$ to hold.
Namely, dimension of a subposet of the Boolean lattice of order $n$ consisting only of singletons and cosingletons (elements of size $1$ and $n-1$, respectively) is already equal to $n$.
Kim, Martin, Masařík, Shull, Smith, Uzzell, and Wang~\cite{KIM2020103074}, and subsequently Lewis~\cite{Lewis20} asked if $\ldim(\lat{n}) = n$ for every $n$.
We answer this question in negative.

\begin{theorem}\label{th:upper_bound}
    For every integer $n$ with $n \geq 4$,
    $$\ldim(\lat{n}) < n.$$
\end{theorem}

Kim et al.~\cite{KIM2020103074} noted that if there exists $n_0$ such that $\ldim(\lat{n_0}) < n_0$, then $\ldim(\lat{n}) < n$ for every $n$ with $n \geq n_0$.
This is due to the subadditivity of the local dimension (see~\Cref{lemma:ldim_prod}) and a product structure of Boolean lattices.
We discuss this in more detail later in the paper.
We show that $\ldim(\lat{4}) \leq 3$ and $\ldim(\lat{7}) \leq 5$.
The former implies~\Cref{th:upper_bound} and the latter implies that $\ldim(\lat{n}) \leq \left\lceil\frac{5}{7} n\right\rceil$ for every positive integer $n$. 

Note that similar methods (computer search aided by SAT-solvers) were recently used by Briański, Hodor, La, Michno, and Micek~\cite{BDIM24} to show that $\bdim(\lat{n}) \leq \left\lceil\frac{5}{6} n\right\rceil$ for every positive integer $n$ where $\bdim(\cdot)$ denotes Boolean dimension, which is another well-studied variant of dimension of posets.
In both cases, it is an interesting challenge to find a more comprehensive proof.

\subsection{The problem of singletons}
The next step in understanding local dimension of a Boolean lattice would be to determine its asymptotic growth.
At least two proofs showing that $\ldim(\lat{n}) = \Omega(\frac{n}{\log n})$ are known: one by Kim et al.~\cite{KIM2020103074} and second by Dam\'{a}sdi, Felsner, Gir\~{a}o, Keszegh, Lewis, Nagy, and Ueckerdt~\cite{Felsner21} (later in the paper, we adapt this proof to solve a different problem, see~\cref{sec:into:multi}).
In fact, the proof of Dam\'{a}sdi et al.\ gives something stronger.
For a positive integer $n$, let \defin{$\lat{n}^1$} be the poset on all nonempty subsets of $[n]$ such that for two such subsets $A$ and $B$, we have $A \leq B$ if and only if $|A| = 1$ and $A \subset B$.
The mentioned proof yields that $\ldim(\lat{n}^1) = \Omega(\frac{n}{\log n})$.
We prove that this bound is asymptotically tight, i.e.\ $\ldim(\lat{n}^1) = \Theta(\frac{n}{\log n})$.

\begin{theorem}\label{th:lat_1_n}
    For every positive integer $n$, 
        \[(1-o(1)) \cdot \frac{n}{\log n} \leq \ldim(\lat{n}^1) \leq 2 \cdot \frac{n}{\log n}+3.\]
\end{theorem}
For the proof of the first inequality, see~\cite[Section~3.1]{Felsner21}.
With slightly worse constants it also follows from~\Cref{lemma:mn_lower_bound} applied with $m=2$.

\subsection{Size of local realizers}
In the study of local dimension it is always emphasized that the size of a local realizer may vastly exceed its frequency.
We are wondering if this is necessarily the case for optimal local realizers.
Our study led us to the following very bold conjecture.

\begin{conjecture}\label{conj}
	There exists a function $f$ such that for every poset $P$, there exists a local realizer $\mathcal{L}$ of $P$ such that the frequency of $\mathcal{L}$ is exactly $\ldim(P)$ and the size of $\mathcal{L}$ is at most $f(\dim(P))$.
\end{conjecture}

We show that if the conjecture holds, then the function $f$ must be superlinear.

\begin{theorem}\label{th:lat_1_n_short_realizer}
    Let $c$ and $n$ be positive integers with $n \geq 2$.
    If $\mathcal{L}$ is a local realizer of $\lat{n}^1$ of size at most $cn$, then the frequency of $\mathcal{L}$ is at least $\frac{n}{2(c+1)}$.
\end{theorem}

Let us comment on why~\Cref{th:lat_1_n_short_realizer} indeed implies that the function $f$ in~\cref{conj} can not be linear.
Suppose to the contrary that~\cref{conj} holds with $f(n) = cn$ for some constant~$c$.
By~\Cref{th:lat_1_n}, we have $\ldim(\lat{n}^1) \leq C \frac{n}{\log n}$.
Certainly, there exists $n$ big enough so that $C \frac{n}{\log n} < \frac{n}{2(c+1)}$.
On the other hand, assumed~\cref{conj} implies that there is an optimal local realizer of size at most $c\dim(\lat{n}^1) = cn$, and so, by~\Cref{th:lat_1_n_short_realizer}, frequency at least $\frac{n}{2(c+1)}$, which is a contradiction.

\subsection{Local dimension of a lattice of multisets}\label{sec:into:multi}
We conclude with a minor complementary result.
Following Briański et al.~\cite{BDIM24}, we consider a natural extension of a Boolean lattice, i.e.\ a \defin{lattice of multisets}.
Namely, for every positive integer $n$, we define the poset \defin{$\mlat{n}$}, as the poset on the family of all multisets containing elements in $[n]$ equipped with the inclusion relation. 
In a multiset, we allow elements to have arbitrary multiplicities; thus, these posets have infinitely many elements. 
However, again, it is easy to see that $\dim(\mlat{n}) = n$ for every $n$.
Briański et al.~\cite{BDIM24} showed that $\bdim(\mlat{n}) = n$ for every positive integer $n$.
We give an analogous result.

\begin{theorem}\label{th:mlat_exact_ldim}
    For every positive integer $n$,
        \[\ldim(\mlat{n}) = n.\]
\end{theorem}

\section{Local dimension of a Boolean lattice}\label{sec:upperbound}

Let $P$ and $Q$ be two posets with ground sets $X$ and $Y$, respectively.
The \defin{product} of $P$ and $Q$, denoted by \defin{$P \times Q$}, is the poset with the ground set $X \times Y$, where for any two pairs $(x_1, y_1), (x_2, y_2) \in X \times Y$, we have $(x_1, y_1) \leq (x_2, y_2)$ in $P \times Q$ if and only if $x_1 \leq x_2$ in $P$ and $y_1 \leq y_2$ in $Q$. 
We define inductively \defin{powers} of a poset.
We let $P^1 = P$ and for an integer $n$ with $n \geq 2$, \defin{$P^n$} is defined as $P^{n-1} \times P$.

Kim et al.\ proved that local dimension is subadditive with respect to product in the following sense.

\begin{lemma}[Theorem 20 in \cite{KIM2020103074}]\label{lemma:ldim_prod}
    For every two posets $P$ and $Q$, we have $\ldim(P \times Q) \leq \ldim(P) + \ldim(Q)$.
\end{lemma}

Therefore, in order to show that $\ldim(\lat{n}) < n$ for every $n$ with $n \geq n_0$, it suffices to prove that $\ldim(\lat{n_0}) < n_0$ (as also noted by Kim at el.).
Therefore, the following lemma yields~\Cref{th:upper_bound}.

\begin{table}
    \scriptsize
    \sffamily
    \centering
    \begin{tabular}{|c|c|c|c|}
        \hline
        \rule[-1.5ex]{0pt}{2pt} $L_1$ \rule{0pt}{3ex} & $L_2$ & $L_3$ & $L_4$ \\
        
        {[}\s\1\s\2\s\3\s\4{]} & {[}\s\1\s\2\s\3\s\4{]} & {[}\s\1\s\2\s\3\s\4{]} & {[}\s\s\s\2\s\s\s\4{]}\\
        {[}\s\1\s\s\s\3\s\4{]} & {[}\s\1\s\2\s\s\s\4{]} & {[}\s\1\s\s\s\3\s\4{]} & {[}\s\1\s\s\s\s\s\4{]}\\
        {[}\s\s\s\2\s\3\s\4{]} & {[}\s\1\s\2\s\3\s\s{]} & {[}\s\1\s\s\s\3\s\s{]} & {[}\s\s\s\s\s\s\s\4{]}\\
        {[}\s\s\s\s\s\3\s\4{]} & {[}\s\1\s\2\s\s\s\s{]} & {[}\s\1\s\2\s\s\s\4{]} & {[}\s\1\s\2\s\3\s\s{]}\\
        {[}\s\1\s\2\s\3\s\s{]} & {[}\s\s\s\2\s\3\s\4{]} & {[}\s\1\s\s\s\s\s\4{]} & {[}\s\1\s\s\s\3\s\s{]}\\
        {[}\s\s\s\2\s\3\s\s{]} & {[}\s\s\s\2\s\3\s\s{]} & {[}\s\1\s\s\s\s\s\s{]} & {[}\s\1\s\2\s\s\s\s{]}\\
        {[}\s\s\s\s\s\3\s\s{]} & {[}\s\s\s\2\s\s\s\4{]} & {[}\s\s\s\2\s\3\s\4{]} & {[}\s\1\s\s\s\s\s\s{]}\\
        {[}\s\1\s\2\s\s\s\4{]} & {[}\s\s\s\2\s\s\s\s{]} & {[}\s\s\s\s\s\3\s\4{]} & {[}\s\s\s\2\s\3\s\s{]}\\
        {[}\s\1\s\2\s\s\s\s{]} & {[}\s\1\s\s\s\3\s\4{]} & {[}\s\s\s\s\s\3\s\s{]} & {[}\s\s\s\2\s\s\s\s{]}\\
        {[}\s\s\s\2\s\s\s\s{]} & {[}\s\s\s\s\s\3\s\4{]} & {[}\s\s\s\2\s\s\s\4{]} & {[}\s\s\s\s\s\3\s\s{]}\\
        {[}\s\1\s\s\s\s\s\4{]} & {[}\s\1\s\s\s\3\s\s{]} & {[}\s\s\s\s\s\s\s\s{]} & \s\s\s\s\s\s\\
        {[}\s\1\s\s\s\s\s\s{]} & {[}\s\s\s\s\s\s\s\4{]} & \s\s\s\s\s\s & \s\s\s\s\s\s\\
        {[}\s\s\s\s\s\s\s\4{]} & {[}\s\s\s\s\s\s\s\s{]} & \s\s\s\s\s\s & \s\s\s\s\s\s\\
        {[}\s\s\s\s\s\s\s\s{]} & \s\s\s\s\s\s & \s\s\s\s\s\s & \s\s\s\s\s\s\\

        \hline
        
    \end{tabular}
    
    \vspace{5mm}
    
    \caption{Linear orders $L_1, L_2, L_3, L_4$ on all subsets of $[4]$ forming the local realizer of $\lat{4}$. Each column corresponds to one partial linear order. The greatest element in an order is the top one.}\label{tab:partial_realizer_4}
\end{table}

\begin{lemma}\label{lem:ldim_of_B_4}
    $\ldim(\lat{4}) \leq 3$.
\end{lemma}

\begin{proof}
    In \cref{tab:partial_realizer_4}, we give partial linear orders $L_1, L_2, L_3, L_4$ on the ground set of $\lat{4}$. 
    We claim that $\{L_1, L_2, L_3, L_4\}$ is a local realizer of $\lat{4}$. 
    The claim can be verified using the Python script provided in \cref{sec:code}.
\end{proof}

For emphasis, let us formally show that~\cref{lemma:ldim_prod} and~\Cref{lem:ldim_of_B_4} indeed imply~\Cref{th:upper_bound}.
For every positive integer $n$, where $n = 4k + r$ for some nonnegative integer $k$ and $r \in \{0,1,2,3\}$, we have $\lat{n} = (\lat{4})^k \times \lat{r}$. Hence, 
    \[\ldim(\lat{n}) \leq k \cdot \ldim(\lat{4}) + \ldim(\lat{r}) \leq k \cdot 3 + r \leq \left\lceil\frac{3}{4} n\right\rceil. \]

As mentioned in the introduction, we were able to improve the bound of $\left\lceil\frac{3}{4} n\right\rceil$ to $\left\lceil\frac{5}{7} n\right\rceil$.
For this, analogously, it suffices to prove that $\ldim(\lat{7}) \leq 5$.

\begin{lemma}\label{lem:ldim_of_B_7}
    $\ldim(\lat{7}) \leq 5$.
\end{lemma}

\begin{proof}
    In \cref{tab:partial_realizer_7} (at the end of the paper), we give partial linear orders $L_i$ for $i \in [7]$ on the ground set of $\lat{7}$. 
    We claim that $\{ L_i : i \in [7]\}$ is a local realizer of $\lat{7}$. 
    The claim can be verified using the Python script provided in \cref{sec:code}.
\end{proof}

Let us conlude this section briefly discussing how we found the local realizers witnessing~\Cref{lem:ldim_of_B_4,lem:ldim_of_B_7}.
Theese local realizers (\cref{tab:partial_realizer_4,tab:partial_realizer_7}) were found using a SAT solver~\cite{SoosNC09}. 
The code used to generate the SAT formulas is available in a public repository.\footnote{\url{https://github.com/Mapet13/Local_dimension_of_a_Boolean_lattice}}
For a given poset $P$, a positive integer $k$, and a positive integer $d$, we build a SAT formula so that the formula is satisfiable if and only if $\ldim(P) \leq d$. 
Moreover, given a satisfying assignment to the formula we can construct a local realizer \(\{L_1, \dots, L_k\}\) of $P$ of frequency $d$.
Introduce variables $x_{A, B, i}$ and $y_{A, B, i}$ for every pair of distinct elements $A, B $ in $P$ and $i \in [k]$.
We will set $A <_{L_i} B$ if $x_{A, B, i}$ is set to $1$ in the satisfying assignment, and $B <_{L_i} A$ if $y_{A, B, i}$ is set to $1$ in the satisfying assignment. 
Next, introduce a variable $z_{A, i}$ for every element $A$ in $P$ and $i \in [k]$.
We will use the element $A$ in $L_i$ if and only if $z_{A, i}$ is set to $1$ in the satisfying assignment. 
The SAT formula is a conjunction of the following conditions.
First, we have to make sure that each $L_i$ is a linear order, that is, for all distinct
$A,B,C$ in $P$ and $i \in [k]$ we add clauses 
\begin{align*}
    z_{A, i} \wedge z_{B, i} \wedge z_{C, i} \wedge x_{A,B,i} \wedge x_{B,C,i} &\Rightarrow x_{A, C, i}\\
    z_{A, i} \wedge z_{B, i} \wedge z_{C, i} \wedge y_{A,B,i} \wedge y_{B,C,i} &\Rightarrow y_{A, C, i}.
\end{align*}
Second, we have to make sure that for all distinct $A$ and $B$ in $P$ if $A \leq B$ then there is at least one $i \in [k]$ that $A <_{L_i} B$ and there is no $j \in [k]$ that $B <_{L_j} A$. 
That is, if $A \leq B$ in $P$, then we add clauses 
\[\bigvee_{i=1}^{k} x_{A, B, i} \ \text{ and } \ \neg y_{A, B, i} \text{ for each } i \in [k].\]
Third, we have to make sure that for all distinct $A$ and $B$ in $P$ if $A$ if incomparable to $B$ in $P$, then there is at least one $i \in [k]$ that $A <_{L_i} B$ and at least one $j \in [k]$ that $B <_{L_j} A$. 
That is, if $A$ and $B$ are indeed incomparable in $P$, then we add clauses 
\[\bigvee_{i=1}^{k} x_{A, B, i} \ \text{ and } \ \bigvee_{i=1}^{k} y_{A, B, i}.\]
Next, we need to make sure that relations between variables $x$, $y$ and $z$ are correct. 
That is, for all distinct $A$ and $B$ in $P$ and for each $i \in [k]$, we add following clauses: 
\begin{itemize}
    \item $x_{A, B, i} \Rightarrow z_{A, i}$ and $x_{A, B, i} \Rightarrow z_{B, i}$ (if $A <_{L_i} B$ then both A and B are used in $L_i$),
    \item $y_{A, B, i} \Rightarrow z_{A, i}$ and $y_{A, B, i} \Rightarrow z_{B, i}$ (if $B <_{L_i} A$ then both A and B are used in $L_i$), 
    \item $z_{A, i} \wedge z_{B, i} \Rightarrow x_{A, B, i} \vee y_{A, B, i}$ (if both A and B are used in $L_i$ then $A <_{L_i} B$ or $B <_{L_i} A$),
    \item $\neg x_{A, B, i} \vee \neg y_{A, B, i}$ (it cannot happen that simultaneously $A <_{L_i} B$ and $B <_{L_i} A$)
\end{itemize}
Finally, we need to make sure that the frequency of the local relizer is at most $d$ (i.e.\ each element $A$ in $P$ is used at most $d$ times in the local realizer). 
That is, for each $(d+1) $-tuple $(c_1,\dots,c_{d+1})$ of elements in $[k]$, we add clauses 
\begin{align*}
    \bigvee_{i=1}^{d+1} \neg z_{A, c_i}. 
\end{align*}

\section{The problem of singletons}

\begin{proof}[Proof of~\Cref{th:lat_1_n}]
    Let $d$, $r$, and $n$ be positive integers such that $r = \lceil n \slash d \rceil$.
    For each $i \in [r]$, let $I_i = \{d \cdot (i-1) + j : j \in [d]\} \cap [n]$.
    Note that $[n]$ is a disjoint union of the sets $I_i$ for $i \in [r]$. 
    Let $i \in [r]$.
    For each subset $J \subset I_i$, we define a partial linear extension $L_{i,J}$ of $\lat{n}^1$ in the following way.
    First, let $\cgA_J$ be the collection of all sets $A \subset [n]$ with $|A| > 1$ such that $I_i \setminus (A \cap I_i) = J$, and let $\calS_J = \{\{x\} : x \in J\}$.
    The ground set of $L_{i,J}$ is $\cgA_J \cup \calS_J$ and the ordering is such that each element of $\cgA_J$ is less than each element of $\calS_J$ in $L_{i,J}$.
    Since for every $\{x\} \in \calS_J$ and for every $A \in \calA_J$, we have $x \notin A$ by definition, a partial linear extension $L_{i,J}$ of $\lat{n}^1$ satisfying the above requirement exists.
    Additionally, let $L$ be any linear extension of $\lat{n}^1$ such that for all $x \in [n]$ and $A \subset [n]$ with $|A| > 1$, we have $\{x\} < A$ in $L$.
    Let $L'$ be a linear extension of $\lat{n}^1$ also satisfying the above condition and additionally, for all distinct $x,y \in [n]$ with $\{x\} < \{y\}$ in $L$, we have $\{y\} < \{x\}$ in $L'$, and for all distinct $A,B \in [n]$ with $|A|,|B|>1$ and $A < B$ in $L$, we have $B < A$ in $L'$.
    We claim that 
        \[\calL = \{L,L'\} \cup \{L_{i,J} : i \in [r], \ J \subset I_i\}\]
    is a local realizer of $\lat{n}^1$.
    It suffices to show that for every $x \in [n]$ and $A \subset [n]$ with $|A| > 1$ such that $x \notin A$ there exists $i \in [r]$ and $J \subset I_i$ such that $A < \{x\}$ in $L_{i,J}$.
    Let $i \in [r]$ be such that $x \in I_i$ and let $J = I_i \setminus (A \cap I_i)$.
    By definition, we obtain $A \in \calA_J$ and $\{x\} \in \calS_J$.
    It follows that $A < \{x\}$ in $L_{i,J}$, as desired.

    Next, we compute the frequency of $\calL$.
    First, let $x \in [n]$.
    Note that if $\{x\} \in L_{i,J}$, then surely $x \in I_i$.
    It follows that $x$ occurs in at most $2 + (2^{|I_i|}-1) \leq 2^d+1$ partial linear extensions in $\calL$.
    Next, let $A \subset [n]$ with $|A| > 1$.
    By definition, for each $i \in [r]$, there is exactly one $J \subset [n]$ with $A \in L_{i,J}$ (it is $J = I_i \setminus (A \cap I_i)$).
    It follows that $A$ occurs in $2 + r$ partial linear extensions in $\calL$.
    Summarizing, the frequency of $\calL$ is at most $\max \{2^d+1,r+2\}$.

    Given a positive integer $n$, we set $d = \lceil \log n - \log \log n\rceil$ and $r = \lceil n \slash d \rceil$
    It follows that
    \begin{align*}
        2^d+ 1 &\leq 2^{\log n - \log \log n + 1}+1 = 2\frac{n}{\log n}+1 \text{ and}\\
        r+2 &\leq \frac{n}{\lceil \log n - \log \log n\rceil} + 3 \leq 2\frac{n}{\log n}+3.
    \end{align*}
    Altogether, we obtain $\ldim(\lat{n}^1) \leq 2\frac{n}{\log n} + 3$, as desired.
\end{proof}

\section{Small realizers for the problem of singletons}

In this section, we prove~\Cref{th:lat_1_n_short_realizer}.
We will use the well-known Turán's theorem.

\begin{theorem}[\cite{Turan}]\label{Turan}
    Let $k$ and $n$ be positive integers.
    If a graph $G$ has more than $(1 - \frac{1}{k}) \cdot \frac{n^2}{2}$ edges, then it contains a complete graph on $k+1$ vertices as a subgraph.
\end{theorem}

\begin{proof}[Proof of~\Cref{th:lat_1_n_short_realizer}]
    Let $\calL$ be a local realizer of $\lat{n}^1$.
    For every $L \in \calL$, let $L'$ be $L$ restricted only to singletons, i.e.\ elements of $\lat{n}^1$ of cardinality $1$.
    We define a graph $G_\calL$ in the following way.
    The vertex set is $[n]$ and two distinct $i,j \in [n]$ are connected by an edge if there exists $L \in \calL$ with $i,j \in L$ such that $i$ and $j$ are two greatest elements in $L'$ (the order among $i$ and $j$ does not matter). 

    Let $I$ be an independent set in $G_\calL$.
    We claim that
    \begin{equation}\label{eq:ind-freq}
        \text{the frequency of $\calL$ is at least $|I|$.}
    \end{equation}
    Let $A = [n] \setminus I$.
    More precisely, we claim that the frequency of $A$ in $\calL$ is at least $|I|$.
    For every $x \in I$, choose $L_x \in \calL$ with $A <_{L_x} \{x\}$.
    To justify the claim, it suffices to show that for all distinct $x$ and $y$ in $I$, we have $L_x \neq L_y$.
    Suppose to the contrary that there are distinct $x$ and $y$ in $I$ with $L_x = L_y$, call this partial linear extension $L$.
    Let $I'$ be the set of $z \in I$ with $z \in L$ and $A <_{L} \{z\}$.
    In particular, $x,y \in I'$ and $|I'| > 2$.
    Since $A = [n] \setminus I$, the two greatest elements in $L'$ are members of $I'$, say $x'$ and $y'$.
    However, this contradicts $I$ being an independent set in $G_\calL$.
    This way, we obtain a contradiction proving~\ref{eq:ind-freq}.

    Suppose that $\calL$ has size $cn$.
    It follows that $G_\calL$ has at most $cn$ edges.
    Thus, the complement of $G_\calL$ has more than
    \[\binom{n}{2} - cn - 1 = \frac{n(n-1-2c) - 2}{2} = \left(1 - \frac{(2c+1)n+2}{n^2} \right) \cdot \frac{n^2}{2}\]
    edges.
    By~\Cref{Turan}, the complement of $G_\calL$ contains a complete graph on $\ell := \frac{n^2}{(2c+1)n+2}+1$ vertices.
    Therefore, $G_\calL$ contains an independent set on $\ell$ vertices and by~\ref{eq:ind-freq}, the frequency of $\calL$ is at least $\ell$.
    To complete the proof, we note that 
        \[\ell = \frac{n^2}{(2c+1)n+2}+1 \geq \frac{n}{(2c+1)+\frac{2}{n}}+1 \geq \frac{n}{2c+2} + 1 > \frac{n}{2(c+1)}.\qedhere\]
\end{proof}

\section{Posets of multisets}\label{sec:multisets}

In this section, we prove \cref{th:mlat_exact_ldim}, that is, for every positive integer $n$, $\ldim(\mlat{n}) = \dim(\mlat{n}) = n$.
As mentioned in the introduction, we have $\ldim(\mlat{n}) \leq \dim(\mlat{n}) = n$.
Therefore, in order to prove \cref{th:mlat_exact_ldim}, it suffices to show that $n \leq \ldim(\mlat{n})$ for every positive integer $n$.
To this end, we analyze subposets of $\mlat{n}$, where multiplicities of elements in multisets are bounded. 

Let $m$ and $n$ be positive integers.
We define \defin{$\mn{n}{m}$} to be the subposet of $\mlat{n}$ induced by all multisets such that every element of a multiset has multiplicity less than $m$. The number of elements in $\mn{n}{m}$ is equal to $m^n$.
Moreover, $\mn{n}{2}$ is isomorphic to~$\lat{n}$.
Next, we define the singleton variant.
Let \defin{$\cgS_{n,m}$} to be the set of all multisets in $\mn{n}{m}$ consisting of exactly one element with a positive multiplicity, and let \defin{$\mn{n}{m}^+$} be the set of all multisets in $\mn{n}{m}$ consisting of at least two elements with a positive multiplicity.
We define \defin{$\mn{n}{m}^1$} to be the poset on the ground set $\cgS_{n,m} \cup \mn{n}{m}^+$ with the order relation such that $A < B$ if and only if $A \in \cgS_{n,m}$, $B \in \mn{n}{m}^+$, and $A \subset B$.
Since clearly $\mn{n}{m}^1$ is a subposet of $\mn{n}{m}$, it suffices to show a lower bound for the former poset.

\begin{lemma}\label{lemma:mn_lower_bound}
   For all positive integers $n$ and $m$, we have $\ldim(\mn{n}{m}^1) \geq \frac{n \log m - \log n}{\log(3n^2m)}$.
\end{lemma}

\begin{proof}
     Let $n, m$ be positive integers and assume that $\ldim(\mn{n}{m}^1) = d$ for some positive integer~$d$, and let $\calL$ be a local realizer of $\mn{n}{m}^1$ of frequency $d$.

    Let $\calL'$ be the set of all partial linear extensions in $\calL$ that contain at least one element of $\calS_{n,m}$. 
    Let $\calL' = \{L_1,\dots,L_t\}$.
    Consider a linear order $M$ on the ground set $\{(A,L_i) : i \in [t], \ A \in L_i\} \cup [t]$.
    The ordering is so that first, $(A,L_i) < (B,L_j)$ if and only if either $i < j$ or $i= j$ and $A < B$ in $L_i$, second, $(A,L_i) < j$ if and only if $i \leq j$, and third, $i < (B,L_j)$ if and only if $i < j$.
    Since the frequency of $\calL$ is $d$, there are at most $d|\calS_{n,m}|$ elements of $\calM_\calS := \{(S,L_i) : S \in \calS_{n,m}, \ i \in [t], \ \{x\} \in L_i\}$.
    Additionally, we obtain that $|\calL'| = t \leq d|\calS_{n,m}|$.

    Let $M_1,\dots,M_k$ be linear orders obtained from removing $\calM_\calS \cup [t]$ from $M$ and taking consecutive intervals of elements in order, i.e.\ for $i,j \in [k]$, if $\alpha \in M_i$ and $\beta \in M_j$, then $\alpha < \beta$ in~$M$.
    Note that $k \leq |\calM_\calS| + t \leq 2d|\calS_{n,m}|$.

    For every $A \in \mn{n}{m}^+$, we define the \defin{signature} of $A$ as $s(A) \in \{0,1\}^k$, where for each $i \in [k]$, we have $s(A)_i = 1$ if and only if $A \in M_i$.
    We claim that for all $A,B \in \mn{n}{m}^+$, if $s(A) = s(B)$, then $A = B$.
    Suppose to the contrary that for two distinct $A$ and $B$ in $\mn{n}{m}^+$, we have $s(A) = s(B)$.
    Since $A$ and $B$ are distinct, there is $x \in [n]$ that has different multiplicities in $A$ and $B$.
    Assume that $x$ has multiplicity $p$ in $A$ and $q$ in $B$, and also $p < q$.
    Let $X$ be the multiset containing only $x$ with multiplicity $q$.
    In particular, $X \in \calS_{n,m}$, $X \subset B$, and $X \not\subset A$. 
    Since $\calL$ is a local realizer of $\mn{n}{m}^1$, there exists $L \in \calL$ with $A <_L X$.
    Note that $X \in L$ implies $L \in \calL'$, and hence, there exists $i \in [t]$ with $L = L_i$.
    Let $j \in [k]$ be the unique index such that $(A,L_i) \in M_j$.
    Since $s(A) = s(B)$, we also have $(B,L_i ) \in M_j$.
    However, this implies that $B <_{L_i} X$, which is a contradiction with $X \subset B$, and $L_i$ being a partial linear extension of $\mn{n}{m}^1$.

    The signature of a given element $A \in \mn{n}{m}^+$ has at least $1$ and at most $d$ nonzero entries.
    It follows that the number of distinct signatures is at most
        \[\sum^d_{j=1} \binom{2dn(m-1)}{j} \leq d\binom{2dn(m-1)}{d} \leq n(2n^2m)^d.\]
    In the above inequalities, we used $d \leq n$, which is true as $d = \ldim(\lat{n}) \leq n$.
    Since $|\mn{n}{m}^+| = m^n - n(m-1)$, and the signature function is injective, we obtain
    \[m^n - n(m-1) \leq n(2n^2m)^d.\]
    Simplifying, $m^n \leq n(2n^2m)^d + n(m-1) \leq n(3n^2m)^d$, and finally, $d \geq \frac{n \log m - \log n}{\log(3n^2m)}$, as desired.
\end{proof}

For every fixed positive integer $n$, the limit $\lim_{m \to \infty}\frac{n \log m - \log n}{\log(3n^2m)}$ is equal to $n$. It follows that $\ldim(\mn{n}{m}) = n$ for a large enough $m$, and so, $\ldim(\mlat{n}) = n$, which concludes the proof of \cref{th:mlat_exact_ldim}.

\bibliographystyle{abbrv}
\bibliography{localdimension}

\newpage

\appendix

\section{The code to verify \texorpdfstring{\cref{lem:ldim_of_B_4,lem:ldim_of_B_7}}{Lemmas}} \label{sec:code}
We provide a short Python code that verifies whether the four linear orders in \cref{tab:partial_realizer_4} form a local realizer of $\lat{4}$, and the seven linear orders in \cref{tab:partial_realizer_7} form a local realizer of $\lat{7}$.
Beneath the code, we give the formatted linear orders so that they can be directly inputted into the script for verification.

\vspace{2mm}

\scriptsize

\begin{tcolorbox}[
    colback=white, 
    colframe=gray,
    sharpish corners
]
\begin{lstlisting}[language=Python]
N = 4 # or 7 if verifying local realizer of B_7 

NOT_FOUND_IDX = -1

def is_subset(A, B):
    return (A | B) == B

def index_in_ple(element, ple):
    for i in range(len(ple)):
        if element == ple[i]:
            return i
    return NOT_FOUND_IDX

def throw_bad_realizer():
    raise Exception(f"Provided set of orders is not a local realizer of B_{N}.")

def verify(orders):
    for A in range(1<<N):
        for B in range(1<<N):
            if A == B:
                continue

            found_a_less_b = False
            found_b_less_a = False

            for ple in orders:
                a_idx = index_in_ple(A, ple)
                b_idx = index_in_ple(B, ple)

                if a_idx == NOT_FOUND_IDX or b_idx == NOT_FOUND_IDX: 
                    continue
                elif a_idx == b_idx:
                    throw_bad_realizer()
                elif a_idx < b_idx:
                    found_a_less_b = True
                elif a_idx > b_idx:
                    found_b_less_a = True

            is_a_less_b = is_subset(A, B)
            is_b_less_a = is_subset(B, A)
            is_a_not_comparable_with_b = not is_a_less_b and not is_b_less_a

            if ((is_a_less_b and (not found_a_less_b or found_b_less_a)) or
                (is_a_not_comparable_with_b and not (found_a_less_b and found_b_less_a))):
                throw_bad_realizer()

    print(f"Provided set of orders is a correct local realizer of B_{N}.")

orders = [list(map(int, input().split())) for i in range(N)]
verify(orders)\end{lstlisting}

\tcblower

Attached file: \textattachfile[color=blue]{verify.py}{verify.py}

\end{tcolorbox}

\vspace{2mm}

\normalsize 

The first input consists of $4$ lines. Each represents one of the linear orders in \cref{tab:partial_realizer_4}. Similarly, the second represents \cref{tab:partial_realizer_7}. The subsets of $[4]$ and $[7]$ are converted to corresponding decimal integers based on their binary representations. For example, the number $13$ corresponds to the set $\{ 1, 3, 4 \}$, since the binary representation of $13$ is $1101_2$.

\scriptsize

\vspace{4mm}

\begin{tcolorbox}[
    colback=white, 
    colframe=gray,
    sharpish corners
]
\begin{lstlisting}[language=Python]
0 8 1 9 2 3 11 4 6 7 12 14 13 15

0 8 5 12 13 2 10 6 14 3 7 11 15

0 10 4 12 14 1 9 11 5 13 15

4 2 6 1 3 5 7 8 9 10
\end{lstlisting}

\tcblower

Attached file: \textattachfile[color=blue]{orders4.in}{orders4.in}

\end{tcolorbox}

\begin{tcolorbox}[
    colback=white, 
    colframe=gray,
    sharpish corners
]
\begin{lstlisting}[language=Python]
 32 1 33 8 40 4 5 36 37 9 12 13 2 10 34 42 6 41 38 7 14 11 15 39 44 45 43 46 47 16 48 24 56 17 18 50 49 20 22 28 21 23 52 60 26 54 30 53 55 25 57 58 62 27 29 59 61 31 63 64 80 68 66 82 84 70 86 65 81 72 73 76 69 77 96 88 89 92 85 91 93 100 112 97 116 95 113 117 104 124 105 121 125 98 106 114 122 107 123 102 103 111 118 126 119 127
 
 0 96 34 98 16 48 50 112 114 4 68 84 36 100 52 38 54 102 8 72 40 12 104 76 44 108 24 60 46 116 118 92 58 62 74 106 78 120 124 90 110 94 122 126 1 65 5 69 67 9 13 73 77 33 97 37 41 101 105 45 11 75 99 107 71 109 79 103 111 17 19 49 51 81 113 83 25 27 89 57 59 91 21 85 121 123 29 93 53 117 61 125 23 55 95
 
 88 2 10 32 104 56 120 42 1 9 25 33 41 57 26 58 3 35 19 73 51 11 27 43 59 74 106 90 122 83 89 75 113 115 105 91 121 107 123 4 68 6 70 20 84 22 86 12 28 5 69 13 76 77 36 100 14 30 92 21 7 85 78 15 79 87 29 31 94 95 37 44 45 38 46 39 47 52 60 54 61 62 63 108 110 101 109 111 116 119 124 127
 
 64 16 8 72 24 4 12 28 1 88 65 17 21 81 84 85 25 89 92 29 93 32 96 48 112 33 36 37 52 116 40 49 113 56 44 53 101 117 41 45 108 109 57 60 61 120 124 121 125 2 34 66 98 18 82 50 114 6 38 70 22 54 86 102 10 42 14 46 110 118 62 122 126 3 35 7 39 15 43 47 19 51 59 55 63 67 99 83 115 123 71 87
 
 16 20 80 2 66 6 70 18 82 22 86 8 14 74 26 30 1 17 90 81 78 9 5 13 94 3 19 11 27 7 15 23 31 67 75 83 91 71 79 87 95 32 40 48 33 34 49 50 35 96 98 97 99 36 38 51 112 100 114 115 102 53 118 39 103 55 119 56 41 104 105 42 57 43 106 46 107 47 110 111 60 58 59 62 63 120 122 123
 
 0 18 3 35 17 19 49 51 64 66 65 67 96 98 80 82 81 83 97 99 4 5 7 68 69 20 21 71 23 37 85 87 113 39 53 55 101 103 115 116 118 117 119 8 12 72 76 24 28 9 73 13 88 77 10 29 89 93 40 44 121 45 14 61 108 124 42 26 109 125 30 11 15 27 31 58 47 63 74 90 78 94 126 75 91 79
 
 0 72 65 73 68 76 69 77 32 97 100 101 104 108 105 2 34 66 10 109 74 106 3 35 67 99 6 70 71 102 78 103 43 75 79 110 107 111 16 18 80 48 50 82 112 114 20 84 22 86 52 24 54 88 117 56 115 23 87 119 25 120 28 92 125 26 90 30 94 126 31 95 127
\end{lstlisting}

\tcblower

Attached file: \textattachfile[color=blue]{orders7.in}{orders7.in}

\end{tcolorbox}

\newpage

\begin{table}[ht]
    \scriptsize
    \sffamily
    \centering
    \hspace*{-1.4cm}
    \begin{tabular}{|c|c|c|c|c|c|c|}
        \hline
        \rule[-1.5ex]{0pt}{2pt} $L_1$ \rule{0pt}{3ex} & $L_2$ & $L_3$ & $L_4$ & $L_5$ & $L_6$ & $L_7$ \\
{[}\s\1\s\2\s\3\s\4\s\5\s\6\s\7{]} & {[}\s\1\s\2\s\3\s\4\s\5\s\s\s\7{]} & {[}\s\1\s\2\s\3\s\4\s\5\s\6\s\7{]} & {[}\s\1\s\2\s\3\s\s\s\5\s\s\s\7{]} & {[}\s\1\s\2\s\s\s\4\s\5\s\6\s\7{]} & {[}\s\1\s\2\s\3\s\4\s\s\s\s\s\7{]} & {[}\s\1\s\2\s\3\s\4\s\5\s\6\s\7{]}\\
{[}\s\1\s\2\s\3\s\s\s\5\s\6\s\7{]} & {[}\s\1\s\2\s\3\s\s\s\5\s\6\s\s{]} & {[}\s\s\s\s\s\3\s\4\s\5\s\6\s\7{]} & {[}\s\1\s\2\s\3\s\s\s\s\s\s\s\7{]} & {[}\s\s\s\2\s\s\s\4\s\5\s\6\s\7{]} & {[}\s\1\s\2\s\s\s\4\s\5\s\s\s\7{]} & {[}\s\1\s\2\s\3\s\4\s\5\s\s\s\7{]}\\
{[}\s\s\s\2\s\3\s\4\s\5\s\6\s\7{]} & {[}\s\1\s\2\s\3\s\s\s\5\s\s\s\s{]} & {[}\s\1\s\2\s\3\s\s\s\5\s\6\s\7{]} & {[}\s\1\s\2\s\s\s\4\s\5\s\6\s\7{]} & {[}\s\s\s\s\s\s\s\4\s\5\s\6\s\7{]} & {[}\s\1\s\2\s\s\s\4\s\s\s\s\s\7{]} & {[}\s\1\s\2\s\3\s\4\s\5\s\s\s\s{]}\\
{[}\s\s\s\2\s\3\s\s\s\5\s\6\s\7{]} & {[}\s\1\s\s\s\3\s\4\s\5\s\6\s\7{]} & {[}\s\s\s\s\s\3\s\s\s\5\s\6\s\7{]} & {[}\s\1\s\2\s\s\s\s\s\5\s\6\s\7{]} & {[}\s\1\s\2\s\3\s\4\s\5\s\6\s\s{]} & {[}\s\s\s\2\s\3\s\4\s\5\s\6\s\7{]} & {[}\s\s\s\2\s\3\s\4\s\5\s\6\s\7{]}\\
{[}\s\1\s\2\s\3\s\4\s\s\s\6\s\7{]} & {[}\s\1\s\s\s\3\s\4\s\5\s\6\s\s{]} & {[}\s\1\s\2\s\3\s\4\s\s\s\6\s\7{]} & {[}\s\1\s\2\s\s\s\s\s\5\s\s\s\7{]} & {[}\s\s\s\2\s\3\s\4\s\5\s\6\s\s{]} & {[}\s\s\s\2\s\3\s\4\s\5\s\s\s\7{]} & {[}\s\s\s\2\s\3\s\4\s\5\s\s\s\7{]}\\
{[}\s\1\s\2\s\3\s\s\s\s\s\6\s\7{]} & {[}\s\1\s\s\s\3\s\s\s\5\s\6\s\7{]} & {[}\s\1\s\s\s\3\s\4\s\s\s\6\s\7{]} & {[}\s\1\s\2\s\s\s\s\s\s\s\6\s\7{]} & {[}\s\1\s\2\s\s\s\4\s\5\s\6\s\s{]} & {[}\s\s\s\2\s\3\s\4\s\s\s\s\s\7{]} & {[}\s\s\s\2\s\3\s\4\s\5\s\s\s\s{]}\\
{[}\s\s\s\2\s\3\s\s\s\s\s\6\s\7{]} & {[}\s\1\s\s\s\3\s\s\s\5\s\6\s\s{]} & {[}\s\1\s\s\s\3\s\s\s\s\s\6\s\7{]} & {[}\s\1\s\2\s\s\s\s\s\s\s\s\s\7{]} & {[}\s\s\s\2\s\s\s\4\s\5\s\6\s\s{]} & {[}\s\s\s\2\s\s\s\4\s\5\s\s\s\7{]} & {[}\s\s\s\2\s\s\s\4\s\5\s\s\s\7{]}\\
{[}\s\1\s\2\s\s\s\4\s\5\s\6\s\7{]} & {[}\s\1\s\s\s\3\s\4\s\5\s\s\s\7{]} & {[}\s\s\s\2\s\3\s\4\s\s\s\6\s\7{]} & {[}\s\1\s\2\s\3\s\4\s\5\s\6\s\s{]} & {[}\s\s\s\s\s\3\s\4\s\5\s\6\s\s{]} & {[}\s\s\s\2\s\s\s\4\s\s\s\s\s\7{]} & {[}\s\s\s\2\s\s\s\4\s\5\s\s\s\s{]}\\
{[}\s\1\s\2\s\s\s\4\s\s\s\6\s\7{]} & {[}\s\1\s\s\s\3\s\4\s\5\s\s\s\s{]} & {[}\s\s\s\s\s\3\s\4\s\s\s\6\s\7{]} & {[}\s\1\s\2\s\3\s\s\s\5\s\6\s\s{]} & {[}\s\1\s\2\s\3\s\4\s\s\s\6\s\7{]} & {[}\s\1\s\2\s\3\s\4\s\5\s\6\s\s{]} & {[}\s\1\s\s\s\3\s\4\s\5\s\6\s\7{]}\\
{[}\s\s\s\2\s\s\s\4\s\5\s\6\s\7{]} & {[}\s\1\s\2\s\s\s\4\s\5\s\6\s\7{]} & {[}\s\1\s\2\s\3\s\4\s\5\s\6\s\s{]} & {[}\s\1\s\2\s\s\s\4\s\5\s\6\s\s{]} & {[}\s\s\s\2\s\3\s\4\s\s\s\6\s\7{]} & {[}\s\1\s\2\s\3\s\4\s\s\s\6\s\s{]} & {[}\s\s\s\s\s\3\s\4\s\5\s\s\s\7{]}\\
{[}\s\s\s\2\s\s\s\s\s\5\s\6\s\7{]} & {[}\s\1\s\s\s\s\s\4\s\5\s\6\s\7{]} & {[}\s\s\s\2\s\3\s\4\s\5\s\6\s\s{]} & {[}\s\1\s\2\s\s\s\s\s\5\s\6\s\s{]} & {[}\s\1\s\2\s\3\s\4\s\s\s\6\s\s{]} & {[}\s\s\s\2\s\s\s\4\s\5\s\6\s\s{]} & {[}\s\s\s\s\s\3\s\4\s\5\s\s\s\s{]}\\
{[}\s\s\s\2\s\s\s\4\s\s\s\6\s\7{]} & {[}\s\1\s\s\s\3\s\s\s\5\s\s\s\7{]} & {[}\s\1\s\s\s\3\s\4\s\5\s\6\s\s{]} & {[}\s\1\s\2\s\s\s\s\s\5\s\s\s\s{]} & {[}\s\1\s\2\s\s\s\4\s\s\s\6\s\7{]} & {[}\s\1\s\2\s\3\s\4\s\5\s\s\s\s{]} & {[}\s\s\s\s\s\s\s\4\s\5\s\6\s\7{]}\\
{[}\s\s\s\2\s\s\s\s\s\s\s\6\s\7{]} & {[}\s\1\s\s\s\3\s\s\s\5\s\s\s\s{]} & {[}\s\s\s\2\s\3\s\s\s\5\s\6\s\s{]} & {[}\s\1\s\2\s\3\s\4\s\s\s\6\s\s{]} & {[}\s\s\s\2\s\3\s\4\s\s\s\6\s\s{]} & {[}\s\1\s\2\s\s\s\4\s\5\s\s\s\s{]} & {[}\s\1\s\s\s\s\s\4\s\5\s\s\s\s{]}\\
{[}\s\1\s\s\s\3\s\4\s\5\s\6\s\7{]} & {[}\s\1\s\2\s\s\s\4\s\5\s\s\s\7{]} & {[}\s\s\s\s\s\3\s\4\s\5\s\6\s\s{]} & {[}\s\1\s\2\s\s\s\4\s\s\s\6\s\s{]} & {[}\s\s\s\2\s\s\s\4\s\s\s\6\s\7{]} & {[}\s\1\s\2\s\3\s\4\s\s\s\s\s\s{]} & {[}\s\1\s\2\s\3\s\s\s\5\s\6\s\7{]}\\
{[}\s\1\s\s\s\s\s\4\s\5\s\6\s\7{]} & {[}\s\1\s\2\s\s\s\4\s\5\s\6\s\s{]} & {[}\s\s\s\s\s\3\s\s\s\5\s\6\s\s{]} & {[}\s\1\s\2\s\3\s\4\s\s\s\s\s\s{]} & {[}\s\1\s\2\s\s\s\4\s\s\s\6\s\s{]} & {[}\s\1\s\2\s\s\s\4\s\s\s\s\s\s{]} & {[}\s\1\s\2\s\3\s\s\s\5\s\s\s\7{]}\\
{[}\s\1\s\s\s\s\s\4\s\s\s\6\s\7{]} & {[}\s\1\s\s\s\s\s\4\s\5\s\6\s\s{]} & {[}\s\1\s\2\s\3\s\4\s\s\s\6\s\s{]} & {[}\s\1\s\2\s\3\s\s\s\s\s\6\s\s{]} & {[}\s\1\s\s\s\s\s\4\s\5\s\6\s\s{]} & {[}\s\s\s\2\s\3\s\4\s\5\s\s\s\s{]} & {[}\s\1\s\2\s\3\s\s\s\5\s\s\s\s{]}\\
{[}\s\s\s\s\s\3\s\4\s\5\s\6\s\7{]} & {[}\s\1\s\s\s\s\s\4\s\5\s\s\s\7{]} & {[}\s\1\s\2\s\3\s\s\s\s\s\6\s\s{]} & {[}\s\1\s\2\s\3\s\s\s\s\s\s\s\s{]} & {[}\s\s\s\2\s\s\s\4\s\s\s\6\s\s{]} & {[}\s\1\s\s\s\3\s\4\s\5\s\6\s\7{]} & {[}\s\1\s\2\s\s\s\s\s\5\s\6\s\7{]}\\
{[}\s\s\s\s\s\s\s\4\s\s\s\6\s\7{]} & {[}\s\1\s\2\s\s\s\4\s\5\s\s\s\s{]} & {[}\s\s\s\2\s\3\s\4\s\s\s\6\s\s{]} & {[}\s\1\s\2\s\s\s\s\s\s\s\6\s\s{]} & {[}\s\1\s\s\s\s\s\4\s\s\s\6\s\7{]} & {[}\s\1\s\s\s\3\s\4\s\s\s\6\s\7{]} & {[}\s\s\s\s\s\s\s\4\s\5\s\6\s\s{]}\\
{[}\s\1\s\s\s\3\s\s\s\5\s\6\s\7{]} & {[}\s\1\s\s\s\s\s\4\s\5\s\s\s\s{]} & {[}\s\s\s\2\s\3\s\s\s\s\s\6\s\s{]} & {[}\s\1\s\2\s\s\s\s\s\s\s\s\s\s{]} & {[}\s\s\s\s\s\s\s\4\s\s\s\6\s\7{]} & {[}\s\s\s\2\s\s\s\4\s\5\s\s\s\s{]} & {[}\s\1\s\s\s\3\s\s\s\5\s\6\s\7{]}\\
{[}\s\1\s\s\s\s\s\s\s\5\s\6\s\7{]} & {[}\s\1\s\2\s\s\s\s\s\5\s\s\s\7{]} & {[}\s\1\s\s\s\3\s\4\s\s\s\6\s\s{]} & {[}\s\s\s\2\s\3\s\4\s\5\s\6\s\7{]} & {[}\s\1\s\s\s\s\s\4\s\s\s\6\s\s{]} & {[}\s\s\s\2\s\s\s\4\s\s\s\6\s\s{]} & {[}\s\s\s\s\s\s\s\4\s\5\s\s\s\7{]}\\
{[}\s\1\s\2\s\3\s\4\s\5\s\s\s\7{]} & {[}\s\1\s\s\s\s\s\s\s\5\s\6\s\7{]} & {[}\s\s\s\s\s\3\s\4\s\s\s\6\s\s{]} & {[}\s\s\s\2\s\s\s\4\s\5\s\6\s\7{]} & {[}\s\s\s\s\s\s\s\4\s\5\s\6\s\s{]} & {[}\s\s\s\s\s\3\s\4\s\5\s\6\s\7{]} & {[}\s\s\s\2\s\3\s\s\s\5\s\6\s\s{]}\\
{[}\s\s\s\s\s\3\s\s\s\5\s\6\s\7{]} & {[}\s\1\s\s\s\s\s\s\s\5\s\s\s\7{]} & {[}\s\1\s\s\s\3\s\s\s\s\s\6\s\s{]} & {[}\s\s\s\2\s\3\s\4\s\5\s\6\s\s{]} & {[}\s\1\s\2\s\3\s\s\s\5\s\6\s\7{]} & {[}\s\s\s\s\s\3\s\4\s\s\s\6\s\7{]} & {[}\s\s\s\s\s\s\s\4\s\5\s\s\s\s{]}\\
{[}\s\1\s\s\s\s\s\s\s\s\s\6\s\7{]} & {[}\s\1\s\2\s\s\s\s\s\5\s\6\s\s{]} & {[}\s\1\s\2\s\3\s\4\s\5\s\s\s\7{]} & {[}\s\s\s\2\s\3\s\s\s\5\s\6\s\7{]} & {[}\s\1\s\2\s\3\s\s\s\5\s\6\s\s{]} & {[}\s\1\s\s\s\3\s\4\s\5\s\6\s\s{]} & {[}\s\s\s\s\s\3\s\s\s\5\s\6\s\s{]}\\
{[}\s\s\s\s\s\s\s\s\s\5\s\6\s\7{]} & {[}\s\1\s\s\s\s\s\s\s\5\s\6\s\s{]} & {[}\s\s\s\2\s\3\s\4\s\5\s\s\s\7{]} & {[}\s\s\s\2\s\3\s\4\s\s\s\6\s\7{]} & {[}\s\1\s\2\s\3\s\s\s\s\s\6\s\7{]} & {[}\s\s\s\2\s\3\s\4\s\s\s\s\s\s{]} & {[}\s\s\s\2\s\3\s\s\s\5\s\s\s\7{]}\\
{[}\s\s\s\s\s\3\s\s\s\s\s\6\s\7{]} & {[}\s\1\s\2\s\s\s\s\s\5\s\s\s\s{]} & {[}\s\1\s\2\s\3\s\4\s\5\s\s\s\s{]} & {[}\s\s\s\2\s\3\s\4\s\s\s\6\s\s{]} & {[}\s\1\s\2\s\3\s\s\s\s\s\6\s\s{]} & {[}\s\1\s\s\s\3\s\4\s\s\s\6\s\s{]} & {[}\s\s\s\2\s\3\s\s\s\5\s\s\s\s{]}\\
{[}\s\1\s\s\s\3\s\4\s\5\s\s\s\7{]} & {[}\s\1\s\s\s\s\s\s\s\5\s\s\s\s{]} & {[}\s\1\s\s\s\3\s\4\s\5\s\s\s\s{]} & {[}\s\s\s\2\s\3\s\4\s\s\s\s\s\s{]} & {[}\s\s\s\2\s\3\s\s\s\5\s\6\s\7{]} & {[}\s\1\s\s\s\s\s\4\s\5\s\6\s\7{]} & {[}\s\s\s\s\s\3\s\s\s\5\s\s\s\7{]}\\
{[}\s\1\s\2\s\s\s\4\s\5\s\s\s\7{]} & {[}\s\1\s\2\s\3\s\4\s\s\s\6\s\7{]} & {[}\s\1\s\2\s\3\s\s\s\5\s\s\s\7{]} & {[}\s\s\s\2\s\s\s\4\s\s\s\6\s\s{]} & {[}\s\1\s\s\s\3\s\s\s\5\s\6\s\s{]} & {[}\s\s\s\s\s\3\s\4\s\s\s\6\s\s{]} & {[}\s\s\s\s\s\3\s\s\s\5\s\s\s\s{]}\\
{[}\s\1\s\s\s\3\s\s\s\5\s\s\s\7{]} & {[}\s\1\s\2\s\3\s\s\s\s\s\6\s\7{]} & {[}\s\1\s\2\s\3\s\4\s\s\s\s\s\7{]} & {[}\s\s\s\2\s\s\s\4\s\s\s\s\s\s{]} & {[}\s\s\s\2\s\3\s\s\s\s\s\6\s\7{]} & {[}\s\s\s\s\s\s\s\4\s\s\s\6\s\s{]} & {[}\s\s\s\2\s\s\s\s\s\5\s\6\s\7{]}\\
{[}\s\s\s\s\s\3\s\4\s\5\s\s\s\7{]} & {[}\s\1\s\2\s\3\s\4\s\s\s\s\s\7{]} & {[}\s\1\s\2\s\3\s\4\s\s\s\s\s\s{]} & {[}\s\s\s\2\s\3\s\s\s\s\s\6\s\7{]} & {[}\s\1\s\2\s\s\s\s\s\5\s\6\s\7{]} & {[}\s\1\s\s\s\3\s\4\s\5\s\s\s\7{]} & {[}\s\s\s\s\s\s\s\s\s\5\s\6\s\7{]}\\
{[}\s\1\s\s\s\s\s\4\s\5\s\s\s\7{]} & {[}\s\1\s\s\s\3\s\4\s\s\s\6\s\7{]} & {[}\s\s\s\2\s\3\s\4\s\s\s\s\s\7{]} & {[}\s\s\s\2\s\3\s\s\s\5\s\s\s\7{]} & {[}\s\s\s\2\s\s\s\s\s\5\s\6\s\7{]} & {[}\s\1\s\s\s\s\s\4\s\5\s\s\s\7{]} & {[}\s\s\s\2\s\s\s\s\s\5\s\s\s\7{]}\\
{[}\s\s\s\s\s\s\s\4\s\5\s\s\s\7{]} & {[}\s\1\s\2\s\3\s\s\s\s\s\s\s\7{]} & {[}\s\1\s\s\s\3\s\s\s\5\s\s\s\7{]} & {[}\s\s\s\2\s\3\s\s\s\5\s\6\s\s{]} & {[}\s\s\s\s\s\3\s\s\s\s\s\6\s\7{]} & {[}\s\1\s\s\s\3\s\4\s\5\s\s\s\s{]} & {[}\s\s\s\2\s\s\s\s\s\5\s\6\s\s{]}\\
{[}\s\s\s\s\s\s\s\s\s\s\s\6\s\7{]} & {[}\s\1\s\2\s\s\s\4\s\s\s\6\s\7{]} & {[}\s\1\s\2\s\3\s\s\s\s\s\s\s\s{]} & {[}\s\s\s\2\s\3\s\s\s\5\s\s\s\s{]} & {[}\s\s\s\s\s\s\s\s\s\5\s\6\s\7{]} & {[}\s\s\s\2\s\s\s\4\s\s\s\s\s\s{]} & {[}\s\s\s\s\s\s\s\s\s\5\s\6\s\s{]}\\
{[}\s\1\s\s\s\3\s\4\s\s\s\s\s\7{]} & {[}\s\1\s\2\s\s\s\s\s\s\s\6\s\7{]} & {[}\s\1\s\s\s\3\s\s\s\5\s\s\s\s{]} & {[}\s\s\s\2\s\3\s\s\s\s\s\s\s\7{]} & {[}\s\1\s\2\s\s\s\s\s\5\s\6\s\s{]} & {[}\s\1\s\s\s\3\s\4\s\s\s\s\s\7{]} & {[}\s\s\s\s\s\s\s\s\s\5\s\s\s\7{]}\\
{[}\s\1\s\s\s\3\s\s\s\s\s\s\s\7{]} & {[}\s\1\s\2\s\s\s\4\s\s\s\s\s\7{]} & {[}\s\s\s\s\s\3\s\4\s\5\s\s\s\7{]} & {[}\s\s\s\2\s\3\s\s\s\s\s\6\s\s{]} & {[}\s\s\s\2\s\3\s\s\s\s\s\6\s\s{]} & {[}\s\s\s\s\s\s\s\4\s\5\s\s\s\7{]} & {[}\s\s\s\2\s\s\s\s\s\5\s\s\s\s{]}\\
{[}\s\s\s\s\s\3\s\4\s\s\s\s\s\7{]} & {[}\s\1\s\2\s\s\s\4\s\s\s\s\s\s{]} & {[}\s\s\s\2\s\3\s\4\s\5\s\s\s\s{]} & {[}\s\s\s\2\s\3\s\s\s\s\s\s\s\s{]} & {[}\s\s\s\s\s\3\s\s\s\s\s\6\s\s{]} & {[}\s\1\s\s\s\3\s\4\s\s\s\s\s\s{]} & {[}\s\s\s\s\s\s\s\s\s\5\s\s\s\s{]}\\
{[}\s\1\s\s\s\s\s\4\s\s\s\s\s\7{]} & {[}\s\1\s\s\s\3\s\4\s\s\s\6\s\s{]} & {[}\s\s\s\2\s\3\s\4\s\s\s\s\s\s{]} & {[}\s\s\s\2\s\s\s\s\s\5\s\6\s\7{]} & {[}\s\1\s\2\s\s\s\s\s\s\s\6\s\7{]} & {[}\s\1\s\s\s\s\s\4\s\s\s\s\s\7{]} & {[}\s\1\s\2\s\3\s\4\s\s\s\6\s\7{]}\\
{[}\s\s\s\s\s\s\s\4\s\s\s\s\s\7{]} & {[}\s\1\s\s\s\s\s\4\s\s\s\6\s\7{]} & {[}\s\s\s\s\s\3\s\s\s\s\s\6\s\7{]} & {[}\s\s\s\2\s\s\s\s\s\5\s\6\s\s{]} & {[}\s\1\s\s\s\s\s\s\s\s\s\6\s\7{]} & {[}\s\1\s\s\s\s\s\4\s\s\s\s\s\s{]} & {[}\s\1\s\2\s\s\s\4\s\s\s\6\s\7{]}\\
{[}\s\1\s\s\s\s\s\s\s\5\s\s\s\7{]} & {[}\s\1\s\s\s\3\s\s\s\s\s\6\s\7{]} & {[}\s\s\s\s\s\3\s\s\s\s\s\6\s\s{]} & {[}\s\s\s\2\s\s\s\s\s\5\s\s\s\7{]} & {[}\s\s\s\2\s\s\s\s\s\s\s\6\s\7{]} & {[}\s\s\s\s\s\3\s\4\s\5\s\s\s\s{]} & {[}\s\s\s\2\s\3\s\4\s\s\s\6\s\7{]}\\
{[}\s\1\s\s\s\s\s\s\s\s\s\s\s\7{]} & {[}\s\1\s\s\s\s\s\4\s\s\s\6\s\s{]} & {[}\s\1\s\s\s\3\s\4\s\s\s\s\s\7{]} & {[}\s\s\s\2\s\s\s\s\s\5\s\s\s\s{]} & {[}\s\s\s\s\s\s\s\s\s\s\s\6\s\7{]} & {[}\s\s\s\s\s\s\s\4\s\5\s\s\s\s{]} & {[}\s\1\s\2\s\3\s\4\s\s\s\s\s\7{]}\\
{[}\s\s\s\2\s\3\s\s\s\5\s\s\s\7{]} & {[}\s\1\s\s\s\3\s\s\s\s\s\6\s\s{]} & {[}\s\s\s\s\s\3\s\4\s\s\s\s\s\7{]} & {[}\s\s\s\2\s\s\s\s\s\s\s\6\s\7{]} & {[}\s\1\s\2\s\s\s\s\s\s\s\6\s\s{]} & {[}\s\s\s\s\s\3\s\4\s\s\s\s\s\7{]} & {[}\s\1\s\2\s\s\s\4\s\s\s\s\s\7{]}\\
{[}\s\s\s\2\s\3\s\s\s\s\s\s\s\7{]} & {[}\s\1\s\s\s\s\s\s\s\s\s\6\s\7{]} & {[}\s\1\s\s\s\3\s\4\s\s\s\s\s\s{]} & {[}\s\s\s\2\s\s\s\s\s\s\s\s\s\7{]} & {[}\s\s\s\2\s\s\s\s\s\5\s\6\s\s{]} & {[}\s\s\s\s\s\s\s\4\s\s\s\s\s\7{]} & {[}\s\1\s\2\s\s\s\4\s\s\s\6\s\s{]}\\
{[}\s\s\s\s\s\3\s\s\s\5\s\s\s\7{]} & {[}\s\1\s\s\s\s\s\s\s\s\s\6\s\s{]} & {[}\s\1\s\s\s\3\s\s\s\s\s\s\s\7{]} & {[}\s\s\s\2\s\s\s\s\s\s\s\6\s\s{]} & {[}\s\1\s\s\s\s\s\s\s\5\s\6\s\s{]} & {[}\s\s\s\s\s\3\s\4\s\s\s\s\s\s{]} & {[}\s\1\s\2\s\3\s\s\s\s\s\6\s\7{]}\\
{[}\s\s\s\2\s\s\s\s\s\5\s\s\s\7{]} & {[}\s\1\s\s\s\3\s\4\s\s\s\s\s\7{]} & {[}\s\1\s\s\s\3\s\s\s\s\s\s\s\s{]} & {[}\s\s\s\2\s\s\s\s\s\s\s\s\s\s{]} & {[}\s\s\s\2\s\s\s\s\s\s\s\6\s\s{]} & {[}\s\s\s\s\s\s\s\4\s\s\s\s\s\s{]} & {[}\s\s\s\2\s\3\s\4\s\s\s\s\s\7{]}\\
{[}\s\s\s\2\s\s\s\s\s\s\s\s\s\7{]} & {[}\s\1\s\s\s\s\s\4\s\s\s\s\s\7{]} & {[}\s\s\s\s\s\3\s\4\s\5\s\s\s\s{]} & {[}\s\1\s\s\s\3\s\4\s\5\s\6\s\7{]} & {[}\s\1\s\s\s\s\s\s\s\s\s\6\s\s{]} & {[}\s\1\s\2\s\3\s\s\s\5\s\6\s\7{]} & {[}\s\s\s\2\s\3\s\s\s\s\s\6\s\7{]}\\
{[}\s\s\s\s\s\3\s\s\s\s\s\s\s\7{]} & {[}\s\1\s\s\s\3\s\4\s\s\s\s\s\s{]} & {[}\s\s\s\s\s\3\s\4\s\s\s\s\s\s{]} & {[}\s\1\s\s\s\s\s\4\s\5\s\6\s\7{]} & {[}\s\s\s\s\s\s\s\s\s\5\s\6\s\s{]} & {[}\s\1\s\s\s\3\s\s\s\5\s\6\s\7{]} & {[}\s\1\s\2\s\3\s\s\s\s\s\s\s\7{]}\\
{[}\s\s\s\s\s\s\s\s\s\5\s\s\s\7{]} & {[}\s\1\s\s\s\s\s\4\s\s\s\s\s\s{]} & {[}\s\s\s\2\s\3\s\s\s\5\s\s\s\7{]} & {[}\s\s\s\s\s\3\s\4\s\5\s\6\s\7{]} & {[}\s\s\s\s\s\s\s\4\s\s\s\6\s\s{]} & {[}\s\s\s\2\s\3\s\s\s\5\s\6\s\7{]} & {[}\s\s\s\2\s\3\s\s\s\s\s\s\s\7{]}\\
{[}\s\s\s\s\s\s\s\s\s\s\s\s\s\7{]} & {[}\s\1\s\2\s\s\s\s\s\s\s\s\s\7{]} & {[}\s\s\s\2\s\3\s\s\s\5\s\s\s\s{]} & {[}\s\s\s\s\s\s\s\4\s\5\s\6\s\7{]} & {[}\s\s\s\s\s\s\s\s\s\s\s\6\s\s{]} & {[}\s\s\s\s\s\3\s\s\s\5\s\6\s\7{]} & {[}\s\s\s\2\s\3\s\s\s\s\s\s\s\s{]}\\
{[}\s\1\s\2\s\3\s\4\s\5\s\6\s\s{]} & {[}\s\1\s\s\s\3\s\s\s\s\s\s\s\7{]} & {[}\s\s\s\s\s\3\s\s\s\5\s\s\s\7{]} & {[}\s\1\s\s\s\3\s\4\s\5\s\6\s\s{]} & {[}\s\1\s\2\s\3\s\4\s\5\s\s\s\7{]} & {[}\s\1\s\2\s\s\s\s\s\5\s\6\s\7{]} & {[}\s\1\s\2\s\s\s\s\s\s\s\6\s\7{]}\\
{[}\s\1\s\2\s\3\s\4\s\5\s\s\s\s{]} & {[}\s\1\s\s\s\3\s\s\s\s\s\s\s\s{]} & {[}\s\s\s\s\s\3\s\s\s\5\s\s\s\s{]} & {[}\s\s\s\s\s\3\s\4\s\5\s\6\s\s{]} & {[}\s\1\s\2\s\3\s\s\s\5\s\s\s\7{]} & {[}\s\1\s\2\s\3\s\s\s\s\s\6\s\7{]} & {[}\s\1\s\2\s\s\s\s\s\s\s\s\s\7{]}\\
{[}\s\1\s\s\s\3\s\4\s\5\s\6\s\s{]} & {[}\s\1\s\s\s\s\s\s\s\s\s\s\s\7{]} & {[}\s\s\s\2\s\3\s\s\s\s\s\s\s\7{]} & {[}\s\1\s\s\s\s\s\4\s\5\s\6\s\s{]} & {[}\s\1\s\2\s\3\s\4\s\s\s\s\s\7{]} & {[}\s\1\s\s\s\3\s\s\s\s\s\6\s\7{]} & {[}\s\1\s\2\s\s\s\s\s\s\s\6\s\s{]}\\
{[}\s\1\s\2\s\s\s\4\s\5\s\6\s\s{]} & {[}\s\1\s\s\s\s\s\s\s\s\s\s\s\s{]} & {[}\s\s\s\2\s\3\s\s\s\s\s\s\s\s{]} & {[}\s\1\s\s\s\3\s\4\s\s\s\6\s\7{]} & {[}\s\1\s\2\s\3\s\s\s\s\s\s\s\7{]} & {[}\s\1\s\2\s\3\s\s\s\5\s\6\s\s{]} & {[}\s\1\s\2\s\s\s\s\s\s\s\s\s\s{]}\\
{[}\s\1\s\s\s\3\s\4\s\5\s\s\s\s{]} & {[}\s\s\s\2\s\3\s\4\s\5\s\6\s\7{]} & {[}\s\s\s\s\s\3\s\s\s\s\s\s\s\7{]} & {[}\s\s\s\s\s\3\s\4\s\s\s\6\s\7{]} & {[}\s\1\s\2\s\s\s\4\s\5\s\s\s\7{]} & {[}\s\1\s\s\s\3\s\s\s\5\s\6\s\s{]} & {[}\s\s\s\2\s\s\s\4\s\s\s\6\s\7{]}\\
{[}\s\1\s\2\s\s\s\4\s\5\s\s\s\s{]} & {[}\s\s\s\2\s\s\s\4\s\5\s\6\s\7{]} & {[}\s\s\s\s\s\3\s\s\s\s\s\s\s\s{]} & {[}\s\1\s\s\s\3\s\4\s\s\s\6\s\s{]} & {[}\s\1\s\2\s\s\s\s\s\5\s\s\s\7{]} & {[}\s\1\s\2\s\3\s\s\s\s\s\6\s\s{]} & {[}\s\s\s\2\s\s\s\4\s\s\s\s\s\7{]}\\
{[}\s\s\s\2\s\3\s\4\s\5\s\6\s\s{]} & {[}\s\s\s\2\s\3\s\4\s\5\s\s\s\7{]} & {[}\s\1\s\2\s\s\s\4\s\5\s\6\s\7{]} & {[}\s\1\s\s\s\s\s\4\s\s\s\6\s\s{]} & {[}\s\1\s\2\s\s\s\4\s\s\s\s\s\7{]} & {[}\s\1\s\s\s\s\s\s\s\5\s\6\s\7{]} & {[}\s\1\s\s\s\3\s\4\s\s\s\6\s\7{]}\\
{[}\s\s\s\2\s\s\s\4\s\5\s\6\s\s{]} & {[}\s\s\s\2\s\3\s\4\s\s\s\6\s\7{]} & {[}\s\1\s\2\s\s\s\4\s\s\s\6\s\7{]} & {[}\s\1\s\s\s\3\s\s\s\5\s\6\s\7{]} & {[}\s\1\s\2\s\s\s\s\s\s\s\s\s\7{]} & {[}\s\1\s\2\s\3\s\s\s\5\s\s\s\7{]} & {[}\s\s\s\2\s\s\s\4\s\s\s\s\s\s{]}\\
{[}\s\1\s\s\s\s\s\4\s\5\s\6\s\s{]} & {[}\s\s\s\2\s\s\s\4\s\5\s\s\s\7{]} & {[}\s\1\s\s\s\s\s\4\s\5\s\6\s\7{]} & {[}\s\1\s\s\s\3\s\s\s\s\s\6\s\7{]} & {[}\s\1\s\2\s\3\s\4\s\5\s\s\s\s{]} & {[}\s\1\s\s\s\3\s\s\s\5\s\s\s\7{]} & {[}\s\s\s\2\s\s\s\s\s\s\s\s\s\7{]}\\
{[}\s\1\s\s\s\s\s\4\s\5\s\s\s\s{]} & {[}\s\s\s\s\s\3\s\4\s\5\s\6\s\7{]} & {[}\s\1\s\2\s\s\s\4\s\5\s\s\s\7{]} & {[}\s\1\s\s\s\3\s\s\s\5\s\6\s\s{]} & {[}\s\1\s\2\s\3\s\s\s\5\s\s\s\s{]} & {[}\s\1\s\s\s\3\s\s\s\s\s\6\s\s{]} & {[}\s\s\s\2\s\s\s\s\s\s\s\6\s\s{]}\\
{[}\s\1\s\2\s\3\s\s\s\5\s\6\s\s{]} & {[}\s\s\s\s\s\s\s\4\s\5\s\6\s\7{]} & {[}\s\1\s\s\s\s\s\4\s\s\s\6\s\7{]} & {[}\s\s\s\s\s\3\s\4\s\s\s\6\s\s{]} & {[}\s\1\s\2\s\3\s\4\s\s\s\s\s\s{]} & {[}\s\1\s\2\s\3\s\s\s\5\s\s\s\s{]} & {[}\s\s\s\2\s\s\s\s\s\s\s\s\s\s{]}\\
{[}\s\1\s\s\s\3\s\s\s\5\s\6\s\s{]} & {[}\s\s\s\2\s\3\s\4\s\s\s\s\s\7{]} & {[}\s\1\s\2\s\s\s\s\s\5\s\6\s\7{]} & {[}\s\s\s\s\s\s\s\4\s\5\s\6\s\s{]} & {[}\s\1\s\2\s\3\s\s\s\s\s\s\s\s{]} & {[}\s\1\s\2\s\3\s\s\s\s\s\s\s\7{]} & {[}\s\1\s\s\s\s\s\4\s\s\s\6\s\7{]}\\
{[}\s\s\s\2\s\3\s\4\s\5\s\s\s\s{]} & {[}\s\s\s\2\s\s\s\4\s\s\s\6\s\7{]} & {[}\s\1\s\s\s\s\s\s\s\5\s\6\s\7{]} & {[}\s\1\s\s\s\s\s\s\s\5\s\6\s\7{]} & {[}\s\1\s\2\s\s\s\4\s\5\s\s\s\s{]} & {[}\s\1\s\s\s\3\s\s\s\5\s\s\s\s{]} & {[}\s\s\s\s\s\3\s\4\s\s\s\6\s\7{]}\\
{[}\s\s\s\2\s\3\s\s\s\5\s\6\s\s{]} & {[}\s\s\s\2\s\s\s\4\s\s\s\s\s\7{]} & {[}\s\1\s\2\s\s\s\4\s\s\s\s\s\7{]} & {[}\s\1\s\s\s\s\s\s\s\5\s\6\s\s{]} & {[}\s\1\s\2\s\s\s\4\s\s\s\s\s\s{]} & {[}\s\s\s\s\s\3\s\s\s\5\s\s\s\s{]} & {[}\s\s\s\s\s\s\s\4\s\s\s\6\s\7{]}\\
{[}\s\s\s\2\s\s\s\4\s\5\s\s\s\s{]} & {[}\s\s\s\2\s\3\s\4\s\5\s\6\s\s{]} & {[}\s\1\s\s\s\s\s\4\s\5\s\s\s\7{]} & {[}\s\s\s\s\s\s\s\4\s\s\s\6\s\s{]} & {[}\s\1\s\2\s\s\s\s\s\5\s\s\s\s{]} & {[}\s\1\s\s\s\3\s\s\s\s\s\s\s\7{]} & {[}\s\1\s\s\s\3\s\s\s\s\s\6\s\7{]}\\
{[}\s\s\s\s\s\3\s\4\s\5\s\6\s\s{]} & {[}\s\s\s\2\s\s\s\4\s\5\s\6\s\s{]} & {[}\s\1\s\2\s\s\s\s\s\5\s\s\s\7{]} & {[}\s\s\s\s\s\3\s\s\s\5\s\6\s\7{]} & {[}\s\1\s\2\s\s\s\s\s\s\s\s\s\s{]} & {[}\s\s\s\s\s\3\s\s\s\s\s\s\s\7{]} & {[}\s\s\s\s\s\3\s\s\s\s\s\6\s\7{]}\\
{[}\s\s\s\s\s\3\s\s\s\5\s\6\s\s{]} & {[}\s\s\s\s\s\3\s\4\s\5\s\s\s\7{]} & {[}\s\s\s\2\s\s\s\4\s\5\s\6\s\7{]} & {[}\s\s\s\s\s\3\s\s\s\5\s\6\s\s{]} & {[}\s\s\s\2\s\3\s\4\s\5\s\s\s\7{]} & {[}\s\1\s\2\s\3\s\s\s\s\s\s\s\s{]} & {[}\s\1\s\s\s\s\s\s\s\s\s\6\s\7{]}\\
{[}\s\1\s\2\s\3\s\s\s\5\s\s\s\s{]} & {[}\s\s\s\2\s\3\s\s\s\5\s\6\s\7{]} & {[}\s\s\s\2\s\s\s\4\s\5\s\s\s\7{]} & {[}\s\1\s\s\s\3\s\s\s\s\s\6\s\s{]} & {[}\s\1\s\s\s\3\s\4\s\s\s\s\s\s{]} & {[}\s\1\s\s\s\3\s\s\s\s\s\s\s\s{]} & {[}\s\s\s\s\s\s\s\s\s\s\s\6\s\s{]}\\
{[}\s\1\s\s\s\3\s\s\s\5\s\s\s\s{]} & {[}\s\s\s\s\s\3\s\s\s\5\s\6\s\7{]} & {[}\s\s\s\2\s\s\s\4\s\s\s\6\s\7{]} & {[}\s\s\s\s\s\3\s\s\s\s\s\6\s\s{]} & {[}\s\1\s\s\s\3\s\s\s\s\s\s\s\s{]} & {[}\s\s\s\s\s\3\s\s\s\s\s\s\s\s{]} & {[}\s\1\s\s\s\3\s\4\s\s\s\s\s\7{]}\\
{[}\s\s\s\s\s\3\s\4\s\5\s\s\s\s{]} & {[}\s\s\s\2\s\3\s\4\s\s\s\6\s\s{]} & {[}\s\s\s\2\s\s\s\4\s\s\s\s\s\7{]} & {[}\s\1\s\s\s\s\s\s\s\s\s\6\s\s{]} & {[}\s\1\s\s\s\s\s\4\s\s\s\s\s\s{]} & {[}\s\1\s\2\s\s\s\s\s\s\s\6\s\7{]} & {[}\s\1\s\s\s\3\s\s\s\s\s\s\s\7{]}\\
{[}\s\s\s\2\s\3\s\s\s\5\s\s\s\s{]} & {[}\s\s\s\s\s\3\s\4\s\5\s\6\s\s{]} & {[}\s\1\s\2\s\s\s\4\s\5\s\6\s\s{]} & {[}\s\s\s\s\s\s\s\s\s\5\s\6\s\7{]} & {[}\s\s\s\2\s\3\s\4\s\s\s\s\s\7{]} & {[}\s\1\s\s\s\s\s\s\s\s\s\6\s\7{]} & {[}\s\s\s\s\s\3\s\4\s\s\s\s\s\7{]}\\
{[}\s\s\s\s\s\3\s\s\s\5\s\s\s\s{]} & {[}\s\s\s\s\s\s\s\4\s\5\s\s\s\s{]} & {[}\s\1\s\2\s\s\s\4\s\s\s\6\s\s{]} & {[}\s\s\s\s\s\s\s\s\s\5\s\6\s\s{]} & {[}\s\1\s\s\s\s\s\s\s\5\s\s\s\7{]} & {[}\s\1\s\2\s\s\s\s\s\5\s\s\s\7{]} & {[}\s\s\s\s\s\3\s\s\s\s\s\s\s\7{]}\\

        {...} & {...} & {...} & {...} & {...} & {...} & {...}\\
        \hline
    \end{tabular}

\end{table}

\begin{table}[]
    \scriptsize
    \sffamily
    \centering
    \hspace*{-1.4cm}
    \begin{tabular}{|c|c|c|c|c|c|c|}
        \hline
        \rule[-1.5ex]{0pt}{2pt} $L_1$ \rule{0pt}{3ex} & $L_2$ & $L_3$ & $L_4$ & $L_5$ & $L_6$ & $L_7$ \\
        {...} & {...} & {...} & {...} & {...} & {...} & {...}\\
{[}\s\1\s\s\s\s\s\s\s\5\s\6\s\s{]} & {[}\s\s\s\s\s\3\s\4\s\s\s\6\s\7{]} & {[}\s\1\s\2\s\s\s\4\s\5\s\s\s\s{]} & {[}\s\s\s\s\s\s\s\s\s\s\s\6\s\7{]} & {[}\s\s\s\2\s\s\s\4\s\5\s\s\s\7{]} & {[}\s\1\s\s\s\s\s\s\s\5\s\s\s\7{]} & {[}\s\1\s\s\s\s\s\4\s\s\s\s\s\7{]}\\
{[}\s\s\s\2\s\s\s\s\s\5\s\6\s\s{]} & {[}\s\s\s\s\s\3\s\4\s\s\s\6\s\s{]} & {[}\s\1\s\2\s\s\s\4\s\s\s\s\s\s{]} & {[}\s\s\s\s\s\s\s\s\s\s\s\6\s\s{]} & {[}\s\1\s\s\s\s\s\s\s\5\s\s\s\s{]} & {[}\s\s\s\2\s\s\s\s\s\5\s\s\s\7{]} & {[}\s\1\s\s\s\s\s\s\s\s\s\s\s\7{]}\\
{[}\s\s\s\2\s\s\s\s\s\5\s\s\s\s{]} & {[}\s\s\s\s\s\3\s\4\s\s\s\s\s\7{]} & {[}\s\1\s\2\s\s\s\s\s\5\s\6\s\s{]} & {[}\s\1\s\s\s\3\s\4\s\5\s\s\s\7{]} & {[}\s\1\s\s\s\s\s\s\s\s\s\s\s\s{]} & {[}\s\s\s\s\s\s\s\s\s\5\s\s\s\7{]} & {[}\s\s\s\s\s\s\s\4\s\s\s\s\s\7{]}\\
{[}\s\1\s\s\s\s\s\s\s\5\s\s\s\s{]} & {[}\s\s\s\s\s\s\s\4\s\s\s\6\s\7{]} & {[}\s\1\s\s\s\s\s\4\s\s\s\s\s\7{]} & {[}\s\1\s\s\s\3\s\4\s\5\s\s\s\s{]} & {[}\s\s\s\2\s\3\s\4\s\5\s\s\s\s{]} & {[}\s\s\s\2\s\s\s\s\s\s\s\6\s\7{]} & {[}\s\s\s\s\s\s\s\s\s\s\s\s\s\s{]}\\
{[}\s\s\s\s\s\s\s\4\s\5\s\6\s\s{]} & {[}\s\s\s\s\s\3\s\4\s\s\s\s\s\s{]} & {[}\s\1\s\2\s\s\s\s\s\5\s\s\s\s{]} & {[}\s\s\s\s\s\3\s\4\s\5\s\s\s\7{]} & {[}\s\s\s\2\s\s\s\4\s\5\s\s\s\s{]} & {[}\s\s\s\s\s\s\s\s\s\s\s\6\s\7{]} & \s\s\s\s\s\s\s\s\s\\
{[}\s\s\s\s\s\s\s\4\s\5\s\s\s\s{]} & {[}\s\s\s\s\s\s\s\4\s\s\s\6\s\s{]} & {[}\s\1\s\2\s\s\s\s\s\s\s\6\s\s{]} & {[}\s\1\s\s\s\s\s\4\s\5\s\s\s\7{]} & {[}\s\s\s\2\s\s\s\4\s\s\s\s\s\7{]} & {[}\s\1\s\2\s\s\s\s\s\s\s\s\s\7{]} & \s\s\s\s\s\s\s\s\s\\
{[}\s\s\s\s\s\s\s\s\s\5\s\6\s\s{]} & {[}\s\s\s\s\s\s\s\4\s\s\s\s\s\7{]} & {[}\s\1\s\2\s\s\s\s\s\s\s\s\s\s{]} & {[}\s\1\s\s\s\s\s\4\s\5\s\s\s\s{]} & {[}\s\s\s\2\s\3\s\4\s\s\s\s\s\s{]} & {[}\s\1\s\s\s\s\s\s\s\s\s\s\s\7{]} & \s\s\s\s\s\s\s\s\s\\
{[}\s\s\s\s\s\s\s\s\s\5\s\s\s\s{]} & {[}\s\s\s\s\s\s\s\4\s\s\s\s\s\s{]} & {[}\s\s\s\2\s\s\s\4\s\5\s\6\s\s{]} & {[}\s\1\s\s\s\3\s\s\s\5\s\s\s\7{]} & {[}\s\s\s\s\s\s\s\4\s\s\s\s\s\s{]} & {[}\s\s\s\2\s\s\s\s\s\s\s\s\s\7{]} & \s\s\s\s\s\s\s\s\s\\
{[}\s\1\s\2\s\3\s\4\s\s\s\6\s\s{]} & {[}\s\s\s\2\s\3\s\s\s\s\s\6\s\7{]} & {[}\s\s\s\2\s\s\s\4\s\5\s\s\s\s{]} & {[}\s\s\s\s\s\3\s\s\s\5\s\s\s\7{]} & {[}\s\s\s\2\s\3\s\s\s\5\s\s\s\7{]} & {[}\s\s\s\s\s\s\s\s\s\s\s\s\s\7{]} & \s\s\s\s\s\s\s\s\s\\
{[}\s\s\s\2\s\3\s\4\s\s\s\6\s\s{]} & {[}\s\s\s\2\s\3\s\s\s\5\s\6\s\s{]} & {[}\s\1\s\s\s\s\s\4\s\5\s\6\s\s{]} & {[}\s\1\s\s\s\s\s\s\s\5\s\s\s\7{]} & {[}\s\s\s\2\s\3\s\s\s\5\s\s\s\s{]} & {[}\s\1\s\2\s\s\s\s\s\5\s\6\s\s{]} & \s\s\s\s\s\s\s\s\s\\
{[}\s\1\s\2\s\s\s\4\s\s\s\6\s\s{]} & {[}\s\s\s\2\s\3\s\s\s\s\s\6\s\s{]} & {[}\s\1\s\s\s\s\s\4\s\s\s\6\s\s{]} & {[}\s\1\s\s\s\3\s\s\s\5\s\s\s\s{]} & {[}\s\s\s\2\s\s\s\s\s\5\s\s\s\7{]} & {[}\s\1\s\s\s\s\s\s\s\5\s\6\s\s{]} & \s\s\s\s\s\s\s\s\s\\
{[}\s\1\s\s\s\3\s\4\s\s\s\6\s\s{]} & {[}\s\s\s\s\s\3\s\s\s\5\s\6\s\s{]} & {[}\s\1\s\s\s\s\s\s\s\s\s\6\s\s{]} & {[}\s\1\s\s\s\s\s\s\s\5\s\s\s\s{]} & {[}\s\s\s\2\s\s\s\s\s\5\s\s\s\s{]} & {[}\s\1\s\2\s\s\s\s\s\5\s\s\s\s{]} & \s\s\s\s\s\s\s\s\s\\
{[}\s\s\s\s\s\3\s\4\s\s\s\6\s\s{]} & {[}\s\s\s\s\s\3\s\s\s\s\s\6\s\7{]} & {[}\s\1\s\s\s\s\s\4\s\5\s\s\s\s{]} & {[}\s\1\s\s\s\s\s\s\s\s\s\s\s\7{]} & {[}\s\s\s\2\s\3\s\s\s\s\s\s\s\7{]} & {[}\s\1\s\s\s\s\s\s\s\5\s\s\s\s{]} & \s\s\s\s\s\s\s\s\s\\
{[}\s\1\s\2\s\3\s\s\s\s\s\6\s\s{]} & {[}\s\s\s\s\s\3\s\s\s\s\s\6\s\s{]} & {[}\s\1\s\s\s\s\s\4\s\s\s\s\s\s{]} & {[}\s\s\s\s\s\s\s\4\s\5\s\s\s\7{]} & {[}\s\s\s\2\s\3\s\s\s\s\s\s\s\s{]} & {[}\s\1\s\2\s\s\s\s\s\s\s\6\s\s{]} & \s\s\s\s\s\s\s\s\s\\
{[}\s\1\s\2\s\3\s\4\s\s\s\s\s\s{]} & {[}\s\s\s\s\s\3\s\s\s\5\s\s\s\7{]} & {[}\s\1\s\s\s\s\s\s\s\s\s\s\s\s{]} & {[}\s\1\s\s\s\s\s\s\s\s\s\s\s\s{]} & {[}\s\s\s\2\s\s\s\s\s\s\s\s\s\7{]} & {[}\s\1\s\2\s\s\s\s\s\s\s\s\s\s{]} & \s\s\s\s\s\s\s\s\s\\
{[}\s\1\s\2\s\s\s\4\s\s\s\s\s\s{]} & {[}\s\s\s\s\s\3\s\s\s\s\s\s\s\7{]} & {[}\s\s\s\2\s\s\s\4\s\s\s\6\s\s{]} & {[}\s\s\s\s\s\3\s\4\s\5\s\s\s\s{]} & {[}\s\s\s\2\s\s\s\s\s\s\s\s\s\s{]} & {[}\s\s\s\2\s\s\s\s\s\5\s\s\s\s{]} & \s\s\s\s\s\s\s\s\s\\
{[}\s\s\s\2\s\3\s\4\s\s\s\s\s\s{]} & {[}\s\s\s\s\s\3\s\s\s\s\s\s\s\s{]} & {[}\s\s\s\s\s\s\s\4\s\5\s\6\s\7{]} & {[}\s\s\s\s\s\3\s\4\s\s\s\s\s\s{]} & {[}\s\s\s\s\s\s\s\s\s\5\s\s\s\7{]} & {[}\s\s\s\s\s\s\s\s\s\s\s\s\s\s{]} & \s\s\s\s\s\s\s\s\s\\
{[}\s\1\s\2\s\3\s\s\s\s\s\s\s\s{]} & {[}\s\s\s\2\s\s\s\s\s\5\s\6\s\7{]} & {[}\s\s\s\s\s\s\s\4\s\5\s\6\s\s{]} & {[}\s\s\s\s\s\3\s\s\s\s\s\s\s\s{]} & {[}\s\s\s\s\s\3\s\s\s\5\s\s\s\s{]} & \s\s\s\s\s\s\s\s\s & \s\s\s\s\s\s\s\s\s\\
{[}\s\s\s\2\s\3\s\s\s\s\s\6\s\s{]} & {[}\s\s\s\s\s\s\s\s\s\5\s\6\s\7{]} & {[}\s\s\s\s\s\s\s\4\s\s\s\6\s\7{]} & {[}\s\s\s\s\s\s\s\4\s\5\s\s\s\s{]} & {[}\s\s\s\s\s\s\s\s\s\5\s\s\s\s{]} & \s\s\s\s\s\s\s\s\s & \s\s\s\s\s\s\s\s\s\\
{[}\s\1\s\s\s\s\s\4\s\s\s\6\s\s{]} & {[}\s\s\s\2\s\s\s\s\s\5\s\6\s\s{]} & {[}\s\s\s\s\s\s\s\s\s\s\s\6\s\s{]} & {[}\s\s\s\s\s\s\s\4\s\s\s\s\s\7{]} & \s\s\s\s\s\s\s\s\s & \s\s\s\s\s\s\s\s\s & \s\s\s\s\s\s\s\s\s\\
{[}\s\s\s\2\s\3\s\s\s\s\s\s\s\s{]} & {[}\s\s\s\s\s\s\s\s\s\5\s\6\s\s{]} & {[}\s\s\s\2\s\s\s\4\s\s\s\s\s\s{]} & {[}\s\s\s\s\s\s\s\4\s\s\s\s\s\s{]} & \s\s\s\s\s\s\s\s\s & \s\s\s\s\s\s\s\s\s & \s\s\s\s\s\s\s\s\s\\
{[}\s\s\s\2\s\s\s\4\s\s\s\6\s\s{]} & {[}\s\s\s\s\s\s\s\s\s\5\s\s\s\s{]} & {[}\s\s\s\2\s\s\s\s\s\s\s\s\s\s{]} & {[}\s\s\s\s\s\s\s\s\s\5\s\s\s\s{]} & \s\s\s\s\s\s\s\s\s & \s\s\s\s\s\s\s\s\s & \s\s\s\s\s\s\s\s\s\\
{[}\s\s\s\2\s\s\s\s\s\s\s\6\s\s{]} & {[}\s\s\s\2\s\s\s\s\s\s\s\6\s\7{]} & {[}\s\s\s\s\s\s\s\4\s\5\s\s\s\7{]} & {[}\s\s\s\s\s\s\s\s\s\s\s\s\s\7{]} & \s\s\s\s\s\s\s\s\s & \s\s\s\s\s\s\s\s\s & \s\s\s\s\s\s\s\s\s\\
{[}\s\s\s\2\s\s\s\4\s\s\s\s\s\s{]} & {[}\s\s\s\2\s\s\s\s\s\s\s\6\s\s{]} & \s\s\s\s\s\s\s\s\s & \s\s\s\s\s\s\s\s\s & \s\s\s\s\s\s\s\s\s & \s\s\s\s\s\s\s\s\s & \s\s\s\s\s\s\s\s\s\\
{[}\s\s\s\2\s\s\s\s\s\s\s\s\s\s{]} & {[}\s\s\s\s\s\s\s\s\s\s\s\6\s\7{]} & \s\s\s\s\s\s\s\s\s & \s\s\s\s\s\s\s\s\s & \s\s\s\s\s\s\s\s\s & \s\s\s\s\s\s\s\s\s & \s\s\s\s\s\s\s\s\s\\
{[}\s\1\s\s\s\3\s\4\s\s\s\s\s\s{]} & {[}\s\s\s\s\s\s\s\s\s\s\s\s\s\s{]} & \s\s\s\s\s\s\s\s\s & \s\s\s\s\s\s\s\s\s & \s\s\s\s\s\s\s\s\s & \s\s\s\s\s\s\s\s\s & \s\s\s\s\s\s\s\s\s\\
{[}\s\s\s\s\s\3\s\4\s\s\s\s\s\s{]} & \s\s\s\s\s\s\s\s\s & \s\s\s\s\s\s\s\s\s & \s\s\s\s\s\s\s\s\s & \s\s\s\s\s\s\s\s\s & \s\s\s\s\s\s\s\s\s & \s\s\s\s\s\s\s\s\s\\
{[}\s\1\s\s\s\s\s\4\s\s\s\s\s\s{]} & \s\s\s\s\s\s\s\s\s & \s\s\s\s\s\s\s\s\s & \s\s\s\s\s\s\s\s\s & \s\s\s\s\s\s\s\s\s & \s\s\s\s\s\s\s\s\s & \s\s\s\s\s\s\s\s\s\\
{[}\s\1\s\s\s\3\s\s\s\s\s\6\s\s{]} & \s\s\s\s\s\s\s\s\s & \s\s\s\s\s\s\s\s\s & \s\s\s\s\s\s\s\s\s & \s\s\s\s\s\s\s\s\s & \s\s\s\s\s\s\s\s\s & \s\s\s\s\s\s\s\s\s\\
{[}\s\s\s\s\s\3\s\s\s\s\s\6\s\s{]} & \s\s\s\s\s\s\s\s\s & \s\s\s\s\s\s\s\s\s & \s\s\s\s\s\s\s\s\s & \s\s\s\s\s\s\s\s\s & \s\s\s\s\s\s\s\s\s & \s\s\s\s\s\s\s\s\s\\
{[}\s\1\s\s\s\3\s\s\s\s\s\s\s\s{]} & \s\s\s\s\s\s\s\s\s & \s\s\s\s\s\s\s\s\s & \s\s\s\s\s\s\s\s\s & \s\s\s\s\s\s\s\s\s & \s\s\s\s\s\s\s\s\s & \s\s\s\s\s\s\s\s\s\\
{[}\s\s\s\s\s\3\s\s\s\s\s\s\s\s{]} & \s\s\s\s\s\s\s\s\s & \s\s\s\s\s\s\s\s\s & \s\s\s\s\s\s\s\s\s & \s\s\s\s\s\s\s\s\s & \s\s\s\s\s\s\s\s\s & \s\s\s\s\s\s\s\s\s\\
{[}\s\s\s\s\s\s\s\4\s\s\s\6\s\s{]} & \s\s\s\s\s\s\s\s\s & \s\s\s\s\s\s\s\s\s & \s\s\s\s\s\s\s\s\s & \s\s\s\s\s\s\s\s\s & \s\s\s\s\s\s\s\s\s & \s\s\s\s\s\s\s\s\s\\
{[}\s\s\s\s\s\s\s\4\s\s\s\s\s\s{]} & \s\s\s\s\s\s\s\s\s & \s\s\s\s\s\s\s\s\s & \s\s\s\s\s\s\s\s\s & \s\s\s\s\s\s\s\s\s & \s\s\s\s\s\s\s\s\s & \s\s\s\s\s\s\s\s\s\\
{[}\s\1\s\s\s\s\s\s\s\s\s\6\s\s{]} & \s\s\s\s\s\s\s\s\s & \s\s\s\s\s\s\s\s\s & \s\s\s\s\s\s\s\s\s & \s\s\s\s\s\s\s\s\s & \s\s\s\s\s\s\s\s\s & \s\s\s\s\s\s\s\s\s\\
{[}\s\1\s\s\s\s\s\s\s\s\s\s\s\s{]} & \s\s\s\s\s\s\s\s\s & \s\s\s\s\s\s\s\s\s & \s\s\s\s\s\s\s\s\s & \s\s\s\s\s\s\s\s\s & \s\s\s\s\s\s\s\s\s & \s\s\s\s\s\s\s\s\s\\
{[}\s\s\s\s\s\s\s\s\s\s\s\6\s\s{]} & \s\s\s\s\s\s\s\s\s & \s\s\s\s\s\s\s\s\s & \s\s\s\s\s\s\s\s\s & \s\s\s\s\s\s\s\s\s & \s\s\s\s\s\s\s\s\s & \s\s\s\s\s\s\s\s\s\\
    
        \hline
        
    \end{tabular}
    
    \vspace{5mm}
    
    \caption{Linear orders $L_1, L_2, L_3, L_4, L_5, L_6, L_7$ on all subsets of $[7]$ forming the local realizer of $\lat{7}$. Each column corresponds to one partial linear order. The greatest element in an order is the top one.}\label{tab:partial_realizer_7}
\end{table}

\end{document}